\documentclass[11pt]{amsart}
\usepackage{amssymb}
\usepackage{amsmath}
\usepackage{graphicx}
\usepackage{amscd}
\usepackage{psfrag}
\usepackage{hyperref}

\newtheorem{theorem}{Theorem}[section]
\newtheorem{lemma}[theorem]{Lemma}
\newtheorem{corollary}[theorem]{Corollary}
\newtheorem{proposition}[theorem]{Proposition}

\theoremstyle{definition}

\newtheorem{remark}[theorem]{Remark}

\newtheorem{step}[theorem]{\hskip-2pt}

\newtheorem{notation}[theorem]{Notation}
\newtheorem{example}[theorem]{Example}

\numberwithin{equation}{section}

\allowdisplaybreaks

\newcommand\boxt[1]{\boxed{\text{\rm #1}\index{boxt}}}
\def\boxit#1{\vbox{\hrule\hbox{\vrule\kern3pt
     \vbox{\kern3pt#1\kern3pt}\kern3pt\vrule}\hrule}}

\newcommand\Nil{{\text{Nil}}}
\newcommand\Sol{{\text{Sol}}}
\newcommand\on[1]{{\rm #1}}

\newcommand\msmall{\scriptstyle}
\newcommand\x{\times}
\newcommand\rx{\rtimes}

\newcommand\wt{\widetilde}

\newcommand\ps{{'\kern-2pt S}}
\newcommand\lpr{{'\kern-2pt}}
\newcommand\lppr{{\kern1pt''\kern-2pt}}
\newcommand\lpe{{\kern1pt '\kern-2pt e}}
\newcommand\lppe{{\kern1pt ''\kern-2pt e}}
\newcommand\lpE{{\kern1pt '\kern-2pt E}}
\newcommand\lppE{{\kern1pt ''\kern-2pt E}}
\newcommand\bs{{\backslash}}
\newcommand\ra{\rightarrow}

\newcommand\lra{\longrightarrow}
\newcommand\hra{\hookrightarrow}

\newcommand\bbe{{\mathbb E}}

\newcommand\bbq{{\mathbb Q}}
\newcommand\bbr{{\mathbb R}}
\newcommand\bbz{{\mathbb Z}}

\newcommand\cals{{\mathcal S}}

\newcommand\calz{{\mathcal Z}}

\newcommand\scirc{\kern -2pt\circ\kern -2pt}

\newcommand\aff{\on{Aff}}
\newcommand\aut{\on{Aut}}

\newcommand\inn{\on{Inn}}
\newcommand\out{\on{Out}}

\newcommand\isom{\on{Isom}}

\newcommand\gl{\on{GL}}

\newcommand\sltr{\on{SL}(2,\mathbb R)}

\newcommand\nil{\on{Nil}}

\newcommand\inv{^{-1}}

\newcommand\Pf{\on{P}\kern-2pt f}

  {\begin{list}{}{%
     \settowidth{\labelwidth}{\textsf{#1}}%
     \setlength{\leftmargin}\labelwidth
     \advance\leftmargin+\labelsep}}
  {\end{list}}

\newcommand{\kbmar}[1]
  {{\kern-4pt$^\spadesuit$}\marginpar{\rightline{\boxed{\text{#1}}}}}

\def\tr{\text{\rm{tr}}}

\def\boxit#1{\vbox{\hrule\hbox{\vrule\kern3pt
     \vbox{\kern3pt#1\kern3pt}\kern3pt\vrule}\hrule}}
\def\boxit#1{\vbox{\hrule\hbox{\vrule\kern3pt
     \vbox{\kern3pt#1\kern3pt}\kern3pt\vrule}\hrule}}

\def\rmk#1{\relax}

\def\multilimits@{\bgroup
  \Let@
  \restore@math@cr
  \default@tag
 \baselineskip\fontdimen10 \scriptfont\tw@
 \advance\baselineskip\fontdimen12 \scriptfont\tw@
 \lineskip\thr@@\fontdimen8 \scriptfont\thr@@
 \lineskiplimit\lineskip
 \vbox\bgroup\ialign\bgroup\hfil$\m@th\scriptstyle{##}$\hfil\crcr}
\def\endSb{\crcr\egroup\egroup\egroup}

\def\im{\on{Im}}
\def\id{\on{id}}
\def\angles#1{\langle{#1}\rangle}

\def\lprS{{\kern2pt '{\kern-2pt}S}}
\def\ddanger{\kern-4pt ${^{^\clubsuit}}$ \marginpar{\centerline{$^{\sqrt{}}$}}}
\def\danger#1{\kern-4pt ${^{^\clubsuit}}$ \marginpar{#1}}
\def\dangers#1{\kern-4pt ${^{^\clubsuit}}$ \marginpar{\boxt{{#1}}}}

\def\bzero{\mathbf 0}

\def\shorteq{\kern-4pt=\kern-4pt}
\newcommand{\pr}{\mbox{\rm pr}}
\def\bfa{\mathbf{a}}
\def\bfb{\mathbf{b}}

\def\bfv{\mathbf{v}}

\def\bfn{\mathbf{n}}

\def\bfx{\mathbf{x}}
\def\bfy{\mathbf{y}}
\def\bfz{\mathbf{z}}

\def\bfe{\mathbf{e}}
\def\bft{\mathbf{t}}
\def\bfzero{{\mathbf 0}}
\def\ord{\mathrm{ord}}
\def\gltr{{\mathrm{GL}}(2,\bbr)}
\def\sltz{{\mathrm{SL}}(2,\bbz)}
\def\gltz{{\mathrm{GL}}(2,\bbz)}
\def\wa{\widetilde A}
\def\wb{\widetilde B}

\def\gamsig{{\Gamma_{\kern-2pt\cals}}}
\def\tgamsig{{\wt\Gamma_{\kern-2pt\cals}}}
\def\hgamsig{{\widehat\Gamma_{\kern-2pt\cals}}}
\def\lgamsig{{\wt\Gamma_{(\cals;q,m_1,m_2)}}}
\def\gammt{\widetilde\Gamma}
\def\Sol{\text{\rm Sol}^3}
\def\Solof{{\text{\rm Sol}_1}^{\kern-2pt 4}}
\def\abar{\bar A}

\def\ahat{\hat A}

\def\bhat{\hat B}

\def\coker{\text{\rm Coker}}

\def\alphabar{\bar\alpha}
\def\betabar{\bar\beta}
\def\bbzphi{\mathbb Z_{\Phi}}
\newcommand\xy[2]{{\left[\begin{matrix} #1 \\ #2\end{matrix}\right]}}
\def\parn{\par\noindent}
\def\matNpm
{\left[\begin{array}{rr}
0& 1\\ \kern-6pt-1&0
\end{array}\right]}
\def\matNmp
{\left[\begin{array}{rr}
0& \kern-6pt-1\\ 1&0
\end{array}\right]}
\def\matNpp
{\left[\begin{array}{rr}
0& 1\\\kern2pt 1&0
\end{array}\right]}
\def\matPpm
{\left[\begin{array}{rr}
1& 0\\ 0 &\kern-6pt-1
\end{array}\right]}
\def\matPmm
{\left[\begin{array}{rr}
\kern-7pt -1& 0\\ 0 &\kern-7pt-1
\end{array}\right]}

\def\vara{\varphi(\bar\alpha)}
\def\varb{\varphi(\bar\beta)}

\newcommand\ii{$i$}
\newcommand\gr[1]{{\tt(#1)}}

\begin{document}

\title{\texorpdfstring{Infra-solvmanifolds of ${{{\text{\rm Sol}}_1}^4}$}{Solof-geometry}}

\author{Kyung Bai Lee}
\address{University of Oklahoma, Norman, OK 73019, U.S.A.}
\email{kblee@math.ou.edu}
\author{Scott Thuong}
\address{Pittsburg State University, Pittsburg, KS 66762, U.S.A.}
\email{sthuong@pittstate.edu}


\date{\today}
\dedicatory{}

\begin{abstract}
The purpose of this paper is to classify all compact manifolds modeled on
the 4-dimensional solvable Lie group ${{\text{\rm Sol}_1}^{\kern-2pt 4}}$, and more generally,
the crystallographic groups of  ${{\text{\rm Sol}_1}^{\kern-2pt 4}}$. 
The maximal compact subgroup of 
$\operatorname{Isom}({{\text{\rm Sol}_1}^{\kern-2pt 4}})$ is
$D_4=\mathbb Z_4\rtimes\mathbb Z_2$.
We shall exhibit an infra-solvmanifold of
${{\text{\rm Sol}_1}^{\kern-2pt 4}}$ whose
holonomy is $D_4$. This implies that all possible holonomy groups do
occur; the trivial group, $\mathbb Z_2$ (5 families),
$\mathbb Z_4$, $\mathbb Z_2\times\mathbb Z_2$ (5 families),  and 
$\mathbb Z_4\rtimes\mathbb Z_2$ (2 families).

\end{abstract}

\maketitle

The 4-dimensional Lie group $\Solof$ is the subgroup of $\gl(3,\bbr)$
 defined as
\[
\Solof= \left\{
\left[
\begin{matrix} 1 &x &z\\
0 &e^u &y\\
0 &0 &1 \end{matrix}  
\right] \Bigg\vert x,y,z,u \in \bbr  \right\}.
\]
The nilradical of $\Solof$ is the 3-dimensional Heisenberg group $\Nil$
(the elements of $\Solof$ with $u=0$).  It has 1-dimensional center (the
elements of $\Solof$ with $x=y=u=0$), and the quotient of $\Solof$ by the
center is isomorphic to $\Sol$.  Recall that both $\Nil$ and $\Sol$ are
model spaces for 3-dimensional geometry.
Let $C$ be a maximal compact subgroup of $\aut(\Solof)$. A cocompact
discrete subgroup
\[ 
\varPi \subset \Solof \rx C
\] 
is a \emph{crystallographic group} of $\Solof$. 
The motivation for this arises from the crystallographic groups of
Euclidean space $\bbr^n$, that is, the cocompact discrete subgroups of
$\isom(\bbr^n)=\bbr^n \rx \on{O}(n,\bbr)$.  In general, the classification
of crystallographic groups of nilpotent Lie groups, or certain well-behaved
solvable Lie groups (such as $\Solof$), is an important question.  For
example, crystallographic groups of $\bbr^n$ are classified for
$n\leq4$. See
\cite{4DBieberbachGroup} for a classification. Dekimpe provides a
classification of crystallographic groups of $4$-dimensional nilpotent Lie
groups in \cite{Dekimpe}. A classification of crystallographic groups of
$\Sol$ is given by K.Y. Ha and J.B. Lee in \cite{HL}.

Since the Bieberbach theorems generalize to $\Solof$ \cite{DLR}, the translation 
subgroup of $\varPi$, $\varPi \cap \Solof$, is of finite index in $\varPi$, and 
is a cocompact discrete subgroup (that is, a lattice) of $\Solof$. Fortunately 
for us, the maximal compact subgroup $C$ is very small. It is
$D_4$, the dihedral group of 8 elements.
Therefore, all crystallographic groups of $\Solof$ are extensions of a
lattice by a subgroup $\Phi$ of the finite group $D_4$. On the other hand,
there are many non-isomorphic lattices, which makes things quite
complicated. We shall classify the crystallographic groups of $\Solof$
(this will include the classification of crystallographic groups of
$\Sol$).

A crystallographic group $\varPi \subset \Solof \rx C$ acts naturally on $
\Solof$; that is, for $(a,\alpha) \in \varPi$, $x \in \Solof$,
$(a,\alpha)\cdot x=a \alpha(x)$. The orbit space of $\Solof$ by the action
of a torsion free crystallographic group $\varPi$, $\varPi \bs \Solof$, is
an \emph{infra-solvmanifold} of $\Solof$. By the generalized Bieberbach
theorems, two infra-solvmanifolds of $\Solof$, say $\varPi \bs \Solof$ and
$\varPi' \bs \Solof$, are (affinely) diffeomorphic precisely when $\varPi$
and $\varPi'$ are isomorphic.  We shall exhibit an infra-solvmanifold of
$\Solof$ with maximal holonomy $D_4$, the largest possible.  This implies
that all possible holonomy groups do occur; the trivial group, $\bbz_2$ (5
families), $\bbz_4$, $\bbz_2\x\bbz_2$ (5 families), and $\bbz_4\rx\bbz_2$
(2 families).
\medskip

This paper is organized as follows.
In section 1, we determine $\aut(\Solof)$, and show the dihedral group
$D_4$ of order 8 is the maximal compact subgroup.

In section 2, we recall the classification of lattices of $\Sol$: all are
isomorphic to $\bbz^2 \rx_\cals \bbz$, for some $\cals\in\sltz$,
$\tr(\cals)>2$.

In section 3, we recall the result of \cite{HL} that any crystallographic
group $Q$ of $\Sol$ can be viewed as an extension
\[
1 \ra \bbz^2 \ra Q \ra \bbzphi \ra 1,
\]
where $\bbzphi$ itself is an extension
$1 \ra \bbz \ra \bbzphi \ra \Phi \ra 1$ for
$\Phi \subset D_4$. Using the results of \cite{HL}, Theorem
\ref{abstract-kernel} classifies all possible abstract kernels
$\varphi: \bbzphi \ra \gltz$.

In section 4, we study the classification of $\Sol$-crystallographic
groups, in a similar fashion to that in \cite{HL}. 
We show an isomorphism between $H_{\varphi}^2(\bbzphi, \bbz^2)$ and 
$H^1(\Phi, \coker(I-\cals))$, which greatly simplifies the calculations 
in \cite{HL}. The list is deferred until section \ref{solof-cryst}. 

In section 5, the classification of $\Solof$-lattices as lifts of
$\Sol$-lattices is given.

In section 6, the main classification theorem of crystallographic groups of
$\Solof$, Theorem \ref{ClassificationSolof-geometry}, is proved. We find 8
categories; some are never torsion free, some are always torsion free, and
some contain mixed cases. We determine this by examining the action of a
crystallographic group on $\Solof$.  This theorem also serves as a
classification of $\Sol$-crystallographic groups, by considering the groups
modulo the center of $\Solof$.

In section 7, we first show that $\Solof$ admits an affine structure. It is
much easier to represent crystallographic groups using this affine
structure.  We exhibit two examples of infra-$\Solof$ manifolds. The first
one is where the lattice is ``non-standard''. The second one is a space
with the maximal holonomy group $D_4$. Both yield non-orientable manifolds.

All calculations were done by the program Mathematica \cite{MA}, and
were  hand-checked.

\section
{\texorpdfstring{The automorphism groups of  $\Sol$ and $\Solof$}
{The automorphism groups of  Sol and Solof}}

\label{section1}
The group $\Sol=\bbr^2\rx\bbr$ has group operation
\begin{align*}
(\bfx,u)(\bfy,v)&=(\bfx+E^u\bfy,u+v) \mbox{, where } E^u=
\left[\begin{matrix} e^{-u} & 0\\ 0 & e^{u}\end{matrix}\right].
\end{align*}
Let $\alpha$ be an automorphism of $\Sol$.
Since $\bbr^2$ is the nilradical (maximal normal nilpotent
subgroup) of  	$\Sol$, $\alpha$ induces an automorphism
$A$ of $\bbr^2$, and hence, also an automorphism $\abar$ of the quotient
$\bbr$.
Thus, there is a homomorphism
\[
\CD
\aut(\Sol) @>>> \aut(\bbr^2)\x\aut(\bbr)\\
\alpha     @>>> (A,\abar)
\endCD
\]

\noindent
The following is known.

\begin{proposition}[{\cite[p.2]{HL}}]
\label{abar-auto-prop}

We have $\aut(\Sol) \cong \Sol \rx (\bbr^+ \x D_4)$, where $D_4$ is the
dihedral group with 8 elements.  Under this isomorphism, $\Sol$ acts as
inner automorphisms, and $(\bbr^+ \x D_4)$ is identified with the group of
matrices
$\bbr^+ \x  D_4= \left\langle k \left[\begin{matrix} 0 & -1 \\ 1 & 0
\end{matrix}\right],
k \left[\begin{matrix} 1& 0 \\ 0 & -1
\end{matrix}\right] \right\rangle$, $k > 0$, 
($k=1$ yields $D_4$), $A \in \bbr^+ \x  D_4$  
acts on $\Sol$ as

\[
A: \left(\xy{x}{y},u\right) \longmapsto \left(A\xy{x}{y}, \abar u\right).
\]
($\abar = +1$ if $A$ is diagonal,
$\abar = -1$ otherwise.)
\end{proposition}

We now turn our attention to $\Solof$, embedded in $\gl(3,\bbr)$ as
\[
\Solof= \left\{
s(x,y,z,u):= \left[
\begin{matrix} 1 &x &z\\
0 &e^u &y\\
0 &0 &1 \end{matrix}  
\right]\ \ \Bigg\vert\ x,y,z,u \in \bbr  \right\}.
\]
By writing every element as a product
$$
\left[
\begin{matrix} 1 &e^u x &z\\
0 &e^u &y\\
0 &0 &1 \end{matrix}
\right]
=
\left[
\begin{matrix} 1 &x &z\\
0 &1 &y\\
0 &0 &1 \end{matrix}
\right]
\left[
\begin{matrix} 1 &0 &0\\
0 &e^u &0\\
0 &0 &1 \end{matrix}
\right]:=\bfx \bfe^u,
$$
we see that $\Solof$ is the semi-direct product $\Nil\rx\bbr$, where
$$
(\bfx,u)\cdot(\bfy,v)=(\bfx\cdot \bfe^u \bfy \bfe^{-u},u+v).
$$

$\Nil$ is the nil-radical of $\Solof$, and the center of $\Nil$,
$\bbr=\{s(0,0,z,0) \vert z\in\bbr\}$, is also the center of $\Solof$.
Evidently, $\Solof/\bbr \cong \Sol$. Thus we have a commuting diagram with
exact rows and columns:
\begin{equation}
\label{solof-to-sol}
\CD
@. 1 @. 1 @. @.\\
@.  @VVV @VVV @. @.\\
@. \bbr =\calz(\Nil) @= \bbr =\calz(\Solof) @. @.\\
@. @VVV @VVV @. @.\\
1 @>>> \Nil @>>> \Solof=\Nil\rx\bbr @>>> \bbr @>>> 1\\
@. @VVV @VVV @| @.\\
1 @>>> \bbr^2 @>>> \Sol=\bbr^2\rx\bbr @>>> \bbr @>>> 1\\
@.  @VVV @VVV @. @.\\
@. 1 @. 1 @. @.
\endCD
\end{equation}
The rows split, but the columns do not.

An automorphism $\hat{\alpha}$ of $\Solof$ induces automorphisms of the
center $\bbr$ and the quotient $\Sol$:
$$
\CD
\aut(\Solof) @>>> \aut(\calz(\Solof))\x\aut(\Sol) @>>> \aut(\bbr)\x \aut(\bbr^2)\x\aut(\bbr) \\
\hat{\alpha}     @>>> (\ahat,\alpha) @>>> (\ahat,A, \abar).
\endCD
$$
Similar to the case of $\Nil$, $\ahat$ is multiplication by $\det(A)$.
Conversely, every automorphism of $\Sol$ induces an automorphism of
$\Solof$, and $\aut(\Sol)$ lifts to a subgroup of $\aut(\Solof)$.  More
specifically, we have

\begin{proposition}
\label{autosol14} 
\begin{align*}
\aut(\Solof) \cong \bbr \rx \aut(\Sol) &\cong \bbr \rx (\Sol \rx (\bbr^+ \x D_4))\\
&\cong (\bbr\x\Sol)\rx (\bbr^+ \x D_4),
\end{align*}
where $\Sol \cong \inn(\Solof)$.
The group $\bbr$ is the kernel of the homomorphism
\[
\aut(\Solof) \ra \aut(\Sol).
\] 
The automorphism $\hat k$, $k\in\bbr$, is given by
\[
\hat k:\
\left[
\begin{matrix} 1 &e^u x &z\\
0 &e^u &y\\
0 &0 &1 \end{matrix}
\right]
\longmapsto
\left[
\begin{matrix} 1 &e^u x &z+ku\\
0 &e^u &y\\
0 &0 &1 \end{matrix}
\right].
\]
This commutes with the inner automorphisms of $\Solof$, and $A\in \bbr^+ \x
D_4$ acts on this $\bbr$ by $^A \hat k=(\ahat\cdot\abar)\cdot \hat k$.
\end{proposition}

\begin{proof}
We have seen that the image of $\aut(\Solof)$ under
\[
\aut(\Solof) \ra \aut(\bbr) \x \aut(\Sol)  \ra \aut(\bbr)\x
\aut(\bbr^2)\x\aut(\bbr)
\]
is determined by its image in $\aut(\bbr^2)$.
On the other hand, $\aut(\Sol)$ lifts back to $\aut(\Solof)$.
Recall the isomorphism $\aut(\Sol)\cong\Sol \rx (\bbr^+ \x D_4)$ given 
in Proposition \ref{abar-auto-prop}.
First, $\Sol\subset\Sol\rx  (\bbr^+ \x D_4)$, corresponding to
the inner automorphisms of $\Sol$ lifts to the inner automorphisms of 
$\aut(\Solof)$.
Note that $\inn(\Solof)=\inn(\Sol)\cong\Sol$.

For the subgroup $\bbr^+\x D_4$ of $\aut(\Sol)$, we have that a diagonal or
off-diagonal matrix $A \in \gl(2,\bbr)$ can be lifted to an automorphism of
$\Solof$:
\[
A=
\kern-4pt
\left[\begin{matrix}
a&b\\ c&d
\end{matrix}\right]
:\
\left[
\begin{matrix} 1 &e^u x &z\\
0 &e^u &y\\
0 &0 &1 \end{matrix}
\right]
\longmapsto
\left[
\begin{matrix} 1
&\kern-7pt e^{\abar u}(a x+b y)
&\kern-7pt
{\scriptstyle\frac12 (a b x^2+2 b c x y+c d y^2+2(a d-b c)z)}\\
0 &e^{\abar u} &(c x+d y)\\
0 &0 &1 \end{matrix}
\right].
\]
	
\noindent
This formula is valid only for the cases when
either $a=d=0$ ($\abar=-1$) or $b=c=0$ ($\abar=+1$).

The kernel of $\aut(\Solof)\ra\aut(\Sol)$ is the group of crossed
homomorphisms $Z^1(\Sol,\bbr)$. Since $\Sol$ acts trivially on the center
$\bbr$, the crossed homomorphisms become genuine homomorphisms, and
\[
Z^1(\Sol,\bbr)=\hom(\Sol,\bbr)=\hom(\bbr,\bbr)=\bbr.
\]
Thus we have a splitting $\aut(\Solof) \cong  \bbr \rx \aut(\Sol).$\qedhere
\end{proof}

\begin{proposition}
\label{maximalCompact}
The dihedral group  $D_4=\left\langle \left[\begin{matrix} 0 & -1 \\ 1 & 0
\end{matrix}\right],
 \left[\begin{matrix} 1& 0 \\ 0 & -1
\end{matrix}\right] \right\rangle$ 
is the maximal compact subgroup of both $\aut(\Sol)$ and $\aut(\Solof)$. 
Furthermore, it is unique up to conjugation.
\end{proposition}

\begin{proof}
The statement on uniqueness follows from  \cite[Theorem 3.1]{mostow}.
\end{proof}

\begin{remark}
Up to the $\bbr=Z^1(\Sol,\bbr)$-factor, $\aut(\Solof)=\aut(\Sol)$, and
we may denote an automorphism in $D_4 \subset \aut(\Solof)$ by a $2 \times
2$ matrix $A$ only (suppressing even $\abar$ and $\ahat$) when there is no
confusion likely.
\end{remark}

\begin{remark}
Both $\Sol$ and $\Solof$ admit a left-invariant metric so that
$\isom(\Sol)=\Sol\rx D_4$ and $\isom(\Solof)=\Solof\rx D_4$.
\end{remark}

\section
{\texorpdfstring
{The Lattices of $\Sol$}
{The Lattices of Sol}}

Let
$
\cals=
\left[
\begin{matrix} \sigma_{11} &\sigma_{12}\\
\sigma_{21} &\sigma_{22} \end{matrix}  \right]
\in\text{SL}(2,\bbz)
$
with $\tr(\cals) >2$. Such a matrix has two positive eigenvalues satisfying
$\lambda + \frac 1 \lambda>0$.
Then we can find a diagonalizing matrix $P \in \gltr$, with $\det(P)=1$,
diagonalizing $\cals$: $PSP \inv = \Delta$.

\begin{notation}
For uniformity of statements, we always take
\begin{align*} \Delta
=\left[
\begin{matrix} \tfrac{1}{\lambda}&0\\
0&\lambda \end{matrix}
\right] \mbox{ with } \tfrac{1}{\lambda}<1<\lambda.
\end{align*}
\end{notation}

With such $P$ and $\Delta$ for $\cals$, the relation $PSP \inv = \Delta$
allows us to embed the semidirect product $\bbz^2 \rx_\cals \bbz$ as a
lattice of $\Sol$,
\begin{align}
\label{def-phi}
\phi: \bbz^2\rx_\cals\bbz &\lra \kern24pt\Sol\\
\notag \left(\xy{x}{y},u\right)
&\longmapsto
\notag \left(P\xy{x}{y},u\ln(\lambda)\right).
\end{align}

It maps the generators as follows:
\begin{align}
\label{sol-iso}
\notag
\bfe_1 =\left(\left[\begin{matrix} 1\\ 0 \end{matrix} \right],0\right) &\longmapsto \bft_1=P \bfe_1\\
\bfe_2=\left(\left[\begin{matrix} 0\\ 1 \end{matrix} \right],0\right) &\longmapsto \bft_2=P \bfe_2\\
\notag
\bfe_3=\left(\left[\begin{matrix} 0\\ 0 \end{matrix} \right],1\right) &\longmapsto \bft_3=(\bfzero,\ln(\lambda)).
\end{align}
We denote image of $\bbz^2\rx_\cals\bbz$ by $
\gamsig=\angles{\bft_1,\bft_2,\bft_3} \subset\Sol,
$ which has relations

\begin{align}
\label{3Lattice-presentation}
\begin{split}
[\bft_1,\bft_2]=1,\ \ \
\bft_3\cdot \bft_1\cdot \bft_3\inv
= \bft_1^{\sigma_{11}}\cdot \bft_2^{\sigma_{21}},\ \ \
\bft_3\cdot \bft_2\cdot \bft_3\inv= \bft_1^{\sigma_{12}}\cdot \bft_2^{\sigma_{22}}.
\end{split}
\end{align}

\begin{notation}
We shall refer to a lattice of $\Sol$ generated by $\bft_1, \bft_2, \bft_3$ of the form
in assignment (\ref{sol-iso}) as a \emph{standard lattice} of $\Sol$.
\end{notation}

Conversely, any lattice of $\Sol$ is isomorphic to such a $\gamsig$
as the following proposition shows. We say $\cals, \cals' \in \sltz$ are
\emph{weakly conjugate} if and only if  $\cals'$ is conjugate, via an
element of $\gltz$, to
$\cals$ or ${\cals} \inv$.

\begin{proposition}[{\cite[Theorem 3.4]{HL}}]
\label{lattice-a=0}
There is a one-one correspondence between the isomorphism classes of
$\Sol$-lattices and the weak-conjugacy classes of
$\cals \in \sltz$ with $\tr(\cals)>2$.
Therefore, any lattice  of $\Sol$ is conjugate to $\gamsig$,
for some $\cals$, by an element of $\aff(\Sol) = \Sol \rx \aut(\Sol)$.
\end{proposition}
\begin{proof}
The isomorphism statement follows from Theorem 3.4 in \cite{HL}.
The conjugacy
statement follows from Theorem \ref{BiebSol} below. This can also be seen by direct
computation, as we do for $\Solof$ lattices in Proposition \ref{solofShape}.  
\qedhere
\end{proof}

\section
{\texorpdfstring
{Compatibility of $\cals$ with automorphisms}
{Compatibility of cals with automorphisms}}

Both $\Sol$ and $\Solof$ are type (R) Lie groups that admit generalizations of Bieberbach's 
theorems for crystallographic groups of $\bbr^n$ \cite{seif-book, DLR}.

\begin{theorem}[{\cite[Theorem 8.3.4 and Theorem 8.4.3]{seif-book}}]
\label{BiebSol}
Let $G$ denote either $\Sol$ or $\Solof$, and $C$ denote a maximal compact subgroup of $\aut(G)$.
{\rm{(1)}} 
For a crystallographic group $\varPi \subset G \rx C$ of $G$,  the translation 
subgroup $\varPi \cap G$ is a lattice of $G$, with $\Phi := \varPi/(\varPi \cap G) \subset C$ finite, the holonomy group.
		
{\rm{(2)}} 
Any isomorphism between two crystallographic groups of $G$ is conjugation
by an element of $\aff(G)=G\rx\aut(G)$.
\end{theorem}

When $G$ is either $\Sol$ or $\Solof$, $C$ is conjugate in $\aut(G)$ to
$D_4$ (Proposition \ref{maximalCompact}). Therefore, we can assume that
$C=D_4$ in either case.
We will see that every $\Solof$-crystallographic group $\varPi \subset
\Solof \rx D_4$ projects to some $\Sol$-crystallographic group $Q \subset
\Sol \rx D_4$ under the natural projection $\Solof \rx D_4 \ra \Sol \rx
D_4$.  Therefore, we first recall the classification of
$\Sol$-crystallographic groups in \cite{HL}.  We use different notation
that is more amenable to lifting to the $\Solof$ case.

\begin{proposition}
\label{puttingToLattice}
Any crystallographic group $Q'$ of $\Sol$ can be conjugated in $\aff(\Sol)$
to $Q \subset \Sol \rx D_4$ so that $Q \cap \Sol = \gamsig$. That is, the
translation subgroup of $Q$ is a standard lattice of $\Sol$, generated by
$\bft_1$, $\bft_2$, and $\bft_3$ as in {\rm (\ref{sol-iso})}. Thus, $Q$ is
generated by $\angles{\bft_1, \bft_2, \bft_3}$, and at most two isometries
of the form $(\bft_1^{a_1} \bft_2^{a_2} \bft_3^{a_3}, A)\in\Sol\rx D_4$,
where $a_i$ are rational numbers.
\end{proposition}

\begin{proof}
This follows from Theorem \ref{BiebSol}.
and Proposition \ref{lattice-a=0}.   \qedhere
\end{proof}

We will assume our $\Sol$--crystallographic group $Q$ is embedded in $\Sol
\rx D_4$ as in Proposition \ref{puttingToLattice}.  Note that $Q \cap
\bbr^2 = \angles{\bft_1, \bft_2} \cong \bbz^2$ is a lattice of $\bbr^2$.
Denote $Q/\angles{\bft_1, \bft_2} $ by $\bbzphi$ so that we have the
commuting diagram:

\begin{equation}
\label{coreClass}
\CD
@. 1 @. 1 @.  @.\\
@. @VVV @VVV @. @.\\
@. \bbz^2 @=  \bbz^2 @. @.\\
@. @VVV @VVV @. @.\\
1 @>>> \gamsig @>>> Q @>>> \Phi @>>> 1\\
@. @VV{/ \bbz^2}V @VV{/ \bbz^2}V @| @.\\
1 @>>> \bbz=\angles{\bft_3} @>>> \bbzphi @>>> \Phi @>>> 1\\
@. @VVV @VVV @. @.\\
@. 1 @. 1 @.  @.\\
\endCD
\end{equation}

To classify $Q$ as extensions of $\bbz^2$ by $\bbzphi$ as in
(\ref{coreClass}), we need all possible \emph{abstract kernels}
\[
\varphi: \bbzphi\lra\gltz.
\]
Now $\bbzphi$ is generated by
$\bft_3$ together with $\bar\alpha=(\bft_3^{a_3},A)$
(with possibly an additional generator $\bar\beta=(\bft_3^{b_3},B)$):
\[
\bbzphi=\angles{\bft_3,\bar\alpha=(\bft_3^{a_3},A),
			\bar\beta=(\bft_3^{b_3},B)}.
\]

Note we only need to consider $\Phi \subset D_4$ up to conjugacy. By
definition, $\varphi(\bft_3)=\cals $. Since $\gamsig=\angles{\bft_1,
\bft_2, \bft_3}$ is embedded in $\Sol$ as in (\ref{def-phi}), as an
automorphism of $\bbz^2=\angles{\bft_1,\bft_2}$,
$\alphabar=(\bft_3^{a_3},A)$ should map by $\varphi$ to
$\varphi(\alphabar)=\cals^{a_3}\wa$, where
\begin{align*}
\cals^{a_3} &= P \inv \Delta^{a_3} P,\\
\wa &= P \inv  A P.
\end{align*}

The action of $A \in D_4$ on $\bbz=\angles{\bft_3}$ in $\bbzphi$ is the
induced action of $A$, $\abar$, on the quotient $\bbr=\Sol/\bbr^2$.  Thus,
if $A\in D_4$ is a diagonal matrix, then $\abar=+1$. Otherwise $\abar=-1$,
see Proposition \ref{abar-auto-prop}.  So, if $\abar=+1$,
$\varphi(\alphabar)$ must commute with $\cals$.  Otherwise,
$\varphi(\alphabar)$ conjugates $\cals$ to its inverse.

Theorem \ref{abstract-kernel} below follows from Theorem 8.2 of \cite{HL}. 
In the proof we explain differences in notation.

\begin{theorem}[{\cite[cf. Theorem 8.2]{HL}}]
\label{abstract-kernel}
The following is a complete list of $\bbzphi$ and homomorphisms
$\varphi:\bbzphi\ra\gltz$ with $\varphi(\bft_3)=\cals$ and
\begin{align*}
\varphi(\bft_3^{a_3},A)&=\cals^{a_3}\wa\\
\varphi(\bft_3^{b_3},B)&=\cals^{b_3}\wb,
\end{align*}
up to conjugation in $\gltz$, that is, change of generators for
$\angles{\bft_1, \bft_2} \cong \bbz^2$.
\begin{enumerate}
		
\item[\gr{0}]
$\Phi$ is trivial,\parn
$\bbzphi=\bbz=\angles{\bft_3}$.\parn

\item[\gr{1}]
$\Phi=\bbz_2$: $A=
\left[\begin{matrix}
1 &0\\ 0& -1
\end{matrix}\right]$,\parn
$\bbzphi=\bbz=\angles{\bft_3,\bar\alpha=(\bft_3^{\frac12},A)}$.\parn
$\bullet$
$\vara=-K$ with $\det(K)=-1$,  $\tr(K)=n>0$, and $\cals=n K+I$.
\item[\gr{2a}]
$\Phi=\bbz_2$: $A=
\left[\begin{matrix}
-1 &0\\ 0& -1
\end{matrix}\right]$,\parn
$\bbzphi=\bbz\x\bbz_2=\angles{\bft_3,\bar\alpha=(\bft_3^{0},A)}$.\parn
$\bullet$
$\vara=A$, $\cals\in\sltz$ with $\tr(\cals)>2$.
\item[\gr{2b}]
$\Phi=\bbz_2$: $A=
\left[\begin{matrix}
-1 &0\\ 0& -1
\end{matrix}\right]$,\parn
$\bbzphi=\bbz=\angles{\bft_3,\bar\alpha=(\bft_3^{\frac12},A)}$.\parn
$\bullet$ $\vara=-K$ with $\det(K)=+1$,  $\tr(K)=n>2$, and $\cals=n K-I$.
\item[\gr{3}]
$\Phi=\bbz_2$: $A=
\left[\begin{matrix}
0 &1\\ 1& 0
\end{matrix}\right]$,\parn
$\bbzphi=\bbz\rx\bbz_2=\angles{\bft_3,\bar\alpha=(\bft_3^{0},A)}$.\parn
$\bullet$ $\vara=A$,
$\cals\in\sltz$ with $\tr(\cals)>2$ and $\sigma_{12}=-\sigma_{21}$.
\item[\gr{3\ii}]
$\Phi=\bbz_2$: $A=
\left[\begin{matrix}
0 &1\\ 1& 0
\end{matrix}\right]$,\parn
$\bbzphi=\bbz\rx\bbz_2=\angles{\bft_3,\bar\alpha=(\bft_3^{0},A)}$.\parn
$\bullet$
$\vara=\left[\begin{matrix}
1&0\\ 0&-1
\end{matrix}\right]$,
$\cals\in\sltz$
with $\tr(\cals)>2$ and  $\sigma_{11}=\sigma_{22}$.
\par\noindent
\item[\gr{4}]
$\Phi=\bbz_4$: $A=
\left[\begin{matrix}
0 &1\\ -1& 0
\end{matrix}\right]$,\parn
$\bbzphi=\bbz\rx\bbz_4=\angles{\bft_3,\bar\alpha=(\bft_3^{0},A)}$.\parn
$\bullet$ $\vara=A$,
$\cals\in\sltz$ with $\tr(\cals)>2$ and symmetric.
\item[\gr{5}]
$\Phi=\bbz_2\x\bbz_2$:
$A=
\left[\begin{matrix}
1 &0\\ 0& -1
\end{matrix}\right]$,
$B=
\left[\begin{matrix}
-1 &0\\ 0& -1
\end{matrix}\right]$,\parn
$\bbzphi=\bbz\x\bbz_2=\angles{\bft_3,\bar\alpha=(\bft_3^{\frac12},A),
\bar\beta=(\bft_3^0,B)}$.\parn
$\bullet$ $\vara=-K$, $\varb=B$
\hfill \gr{1}$+$\gr{2a}\parn
$\bullet$
$\cals=n K+I$, $K\in\gltz$, $\det(K)=-1$, and  $\tr(K)=n>0$.
\item[\gr{6a}]
$\Phi=\bbz_2\x\bbz_2$:
$A=
\left[\begin{matrix}
0 &1\\ 1& 0
\end{matrix}\right]$,
$B=
\left[\begin{matrix}
-1 &0\\ 0& -1
\end{matrix}\right]$,\parn
$\bbzphi=(\bbz\x\bbz_2)\rx\bbz_2=\angles{\bft_3,\bar\alpha=(\bft_3^{0},A),
\bar\beta=(\bft_3^0,B)}$.\parn
$\bullet$ $\vara=A$, $\varb=B$ \hfill \gr{3}$+$\gr{2a}\parn
$\bullet$
$\cals\in\sltz$ with $\tr(\cals)>2$ and $\sigma_{12}=-\sigma_{21}$.
\item[\gr{6a\ii}]
$\Phi=\bbz_2\x\bbz_2$:
$A=
\left[\begin{matrix}
0 &1\\ 1& 0
\end{matrix}\right]$,
$B=\def\tr{\text{\rm{tr}}}
\left[\begin{matrix}
-1 &0\\ 0& -1
\end{matrix}\right]$,\parn
$\bbzphi=(\bbz\x\bbz_2)\rx\bbz_2=\angles{\bft_3,\bar\alpha=(\bft_3^{0},A),
\bar\beta=(\bft_3^0,B)}$.\parn
$\bullet$ $\vara=
\left[\begin{matrix}
1 &0\\ 0& -1
\end{matrix}\right]$,\
$\varb=B$
\hfill \gr{3\ii}$+$\gr{2a}\parn
$\bullet$
$\cals\in\sltz$
with $\tr(\cals)>2$ and $\sigma_{11}=\sigma_{22}$.\parn
\item[\gr{6b}]
$\Phi=\bbz_2\x\bbz_2$:
$A=
\left[\begin{matrix}
0 &1\\ 1& 0
\end{matrix}\right]$,
$B=
\left[\begin{matrix}
-1 &0\\ 0& -1
\end{matrix}\right]$,\parn
$\bbzphi=\bbz\rx\bbz_2=\angles{\bft_3,\bar\alpha=(\bft_3^{0},A),
\bar\beta=(\bft_3^{\frac12},B)}$.\parn
$\bullet$ $\vara=A$, $\varb=-K$ \hfill \gr{3}$+$\gr{2b}\parn
$\bullet$
$\cals=n K-I$, where $K\in\sltz$ with $\tr(K)= n>2$; $k_{12}=-k_{21}$.
\item[\gr{6b\ii}]
$\Phi=\bbz_2\x\bbz_2$:
$A=
\left[\begin{matrix}
0 &1\\ 1& 0
\end{matrix}\right]$,
$B=
\left[\begin{matrix}
-1 &0\\ 0& -1
\end{matrix}\right]$,\parn
$\bbzphi=\bbz\rx\bbz_2=\angles{\bft_3,\bar\alpha=(\bft_3^{0},A),
\bar\beta=(\bft_3^{\frac12},B)}$.\parn
$\bullet$ $\vara=\left[\begin{matrix}
1 &0\\ 0& -1
\end{matrix}\right]$, $\varb=-K$ \hfill \gr{3\ii}$+$\gr{2b}\parn
$\bullet$
$\cals=n K-I$, where $K\in\sltz$ with $\tr(K)= n>2$;\
$k_{11}=k_{22}$.\parn
\item[\gr{7}]
$\Phi=\bbz_4\rx\bbz_2$:
$A=\matNpp, B=\matPpm,$\parn
$\bbzphi=(\bbz\x\bbz_2)\rx\bbz_2=\angles{\bft_3,\bar\alpha=(\bft_3^{0},A),
\bar\beta=(\bft_3^{\frac12},B)}$.\parn
$\bullet$
$\vara=A$, $\varb=-K$ \hfill (includes \gr{6a})\qquad  \gr{3}$+$\gr{1}\parn
$\bullet$
$\cals=n K+I$, $K\in\gltz$, $\det(K)=-1$, $\tr(K)>0$;
$k_{12}=-k_{21}$.\parn
\item[\gr{7\ii}]
$\Phi=\bbz_4\rx\bbz_2$:
$A=\matNpp, B=\matPpm$,\parn
$\bbzphi=(\bbz\x\bbz_2)\rx\bbz_2=\angles{\bft_3,\bar\alpha=(\bft_3^{0},A),
\bar\beta=(\bft_3^{\frac12},B)}$.\parn
$\bullet$
$\vara=\matPpm$, $\varb =-K$
\hfill (includes \gr{6a\ii})\qquad \gr{3\ii}$+$\gr{1}\parn
$\bullet$
$\cals=n K+I$, $K\in\gltz$, $\det(K)=-1$,  $\tr(K)=n>0$, $k_{11}=k_{22}$.\parn
\end{enumerate}
\end{theorem}
\begin{proof}

The 9 families of $\Sol$-crystallographic groups in Theorem 8.2 of
\cite{HL} are labeled $E_0$, $E_1$, $E^{\pm}_2$, $E_3$, $E_5$, $E_8$,
$E_9$, $E_{10}$, and $E_{11}$.  The table below shows our notation
convention:
\medskip

\centerline{
\begin{tabular}{| p{.8cm} | p{.8cm} | p{.8cm} | p{.8cm} | p{.8cm} | p{.8cm} | p{.8cm} | p{.8cm} | p{.8cm} | p{.8cm} |}
\hline
$E_0$ & $E_1$ & $E^{+}_2$ & $E^{-}_2$  & $E_3$ & $E_5$ & $E_8$ & $E_9$ & $E_{10}$  & $E_{11}$
\\ \hline
\gr{0} & \gr{2a} & \gr{2b} & \gr{1} & \gr{3}, \newline \gr{3\ii} & \gr{5} & \gr{6a}, \newline \gr{6a\ii} & \gr{6b}, \newline \gr{6b\ii} & \gr{4}  &\gr{7}, 
\newline \gr{7\ii} \\ \hline 
\end{tabular}}

From Theorem 8.2 of \cite{HL}, $\varphi(\alphabar)=\varphi(\bft_3^{0},A)$
where $A=\matNpp \in D_4$ can act on $\angles{\bft_1, \bft_2} \cong \bbz^2$
in two different ways: either $P \inv A P = \matNpp$ or $P \inv A P =
\matPpm$.

In Theorem 8.2 of \cite{HL}, cases $E_3$, $E_8$, $E_{9}$, and $E_{11}$
contain such a holonomy element, and therefore we split each into two
cases, depending on how $\varphi(\alphabar)$ acts on $\angles{\bft_1,
\bft_2} \cong \bbz^2$. We will see that one case always lifts to
crystallographic groups of $\Solof$ with torsion, whereas the other can
lift to torsion free crystallographic groups.

When $\alphabar = (\bft_3^{\frac12}, A)$, $A$ is necessarily diagonal of
order 2, and $\varphi(\alphabar) = P \inv \Delta^{\frac12} A P = -K$, where
$(-K)^2=K^2 = \cals$.  Letting $n=\tr(K)$, it follows that $\cals = nK+I$
when $ \det(K)=-1$, and $\cals = nK-I$ when $\det(K)=1$. This applies to
the cyclic holonomy cases \gr{1}, \gr{2b}.

When $\alphabar = (\bft_3^{0}, A)$,  $\varphi(\alphabar) = P \inv A P$. 
If $A=-I$, $\varphi(\alphabar)=P(-I)P\inv = -I$ (regardless
of $\cals$ and $P$). For other choices of $A$, we have:

\noindent
(1) $P\inv \matNpm P = \matNpm$ if and only if
$P=\pm \left[\begin{array}{rr}
\cos t&\sin t\\ \kern-6pt-\sin t& \cos t
\end{array}\right]$.\\
$\cals$ is diagonalized by such a $P$ if and only if 
$\sigma_{12}=\sigma_{21}$.

\noindent
(2) $P\inv \matNpp P = \matNpp$ if and only if
$P=\pm\left[\begin{matrix}
\cosh t&\sinh t\\ \sinh t& \cosh t
\end{matrix}\right]$.\\
$\cals$ is diagonalized by such a $P$ if and only if $\sigma_{12}=-\sigma_{21}$.
 
\noindent
(3) $P\inv \matNpp P = \matPpm$ if and only if
$P=\pm \frac{1}{\sqrt{2}}\left[\begin{array}{rr}
t& -\frac{1}{t}\\ t &\frac{1}{t}
\end{array}\right]$, $t\not=0$.\\
$\cals$ is diagonalized by such a $P$ if and only if $\sigma_{11}=\sigma_{22}$.

This applies to the cyclic holonomy cases \gr{3}, \gr{3\ii}, and \gr{4}, and forces the stated conditions on $\cals$. The two generator cases follow from the cyclic cases.
\qedhere
\end{proof}

\section
{\texorpdfstring
{Crystallographic groups of $\Sol$}
{Crystallographic groups of Sol}}

With a fixed abstract kernel $\varphi:\bbzphi\ra\gltz$ from Theorem \ref{coreClass}, 
the set of all
equivalence classes of extensions $Q$ in (\ref{coreClass})
is in one-one correspondence with the group $H_{\varphi}^2(\bbzphi, \bbz^2)$. 
The following theorem greatly simplifies the computations in \cite{HL}.

\begin{theorem}
\label{H1-cokernel-theorem}
For each homomorphism $\varphi: \bbzphi\ra \on{GL}(2,\bbz)$, in {\rm Theorem \ref{abstract-kernel}},
we have an
isomorphism $H_{\varphi}^2(\bbzphi; \bbz^2) \cong H^1(\Phi;\coker(I-\cals))$
where 
\[
\coker(I-\cals)\cong (I-\cals)\inv\bbz^2/\bbz^2\subset T^2
\]
is a finite abelian group.
So, the set of
equivalence classes of extensions $Q$,
\[
1\lra \bbz^2\lra Q\lra \bbzphi\lra 1,
\]
is in one-one correspondence
with $H^1(\Phi;\coker(I-\cals))$.
\end{theorem}

\begin{proof}
Since $\det(I-\cals)\not=0$,  $H^1(\Phi;\coker(I-\cals))$ is finite, as $\coker(I-\cals)$
is finite.	First, we verify that $\varphi(\bbzphi)\subset \on{GL}(2,\bbz)=\aut(\bbz^2)$
leaves the group  $(I-\cals)\inv\bbz^2\subset \bbr^2$
containing $\bbz^2$ invariant.
Suppose there exists $\bfa\in\bbr^2$ such that $(I-\cals)\bfa=\bfz\in\bbz^2$. Then,
\begin{align*}
(I-\cals)\big(\varphi(\bft_{3})\bfa\big) = (I-\cals)\big(\cals\bfa\big) 
= \cals\big((I-\cals)(\bfa)\big) = \cals(\bfz)\in\bbz^2.
\end{align*}
Now for $\varphi(\alphabar)$, if $\abar=+1$,
\begin{align*}
(I-\cals)\big(\varphi(\bar \alpha)\bfa\big) =
\varphi(\bar \alpha)(I-\cals)\bfa = \varphi(\bar \alpha)\bfz\in\bbz^2;
\end{align*}
and if $\abar=-1$,  then $\varphi(\alpha)$ conjugates $\cals$ to $\cals \inv$, and so,
\begin{align*}
(I-\cals)\big(\varphi(\bar \alpha)\bfa\big)
=\varphi(\bar \alpha)(-\cals^{-1})(I-\cals)\bfa = \varphi(\bar \alpha)(-\cals^{-1})\bfz\in\bbz^2.
\end{align*}
This shows that, if $\bfa\in (I-\cals)\inv\bbz^2$, then
so are $\varphi(\bft_{3})\bfa$ and $\varphi(\bar \alpha)\bfa$.
Consequently,  $(I-\cals)\inv\bbz^2$ is $\varphi(\bbzphi)$-invariant.
Since $\bfa-\varphi(\bft_{3})\bfa=(I -\cals)\bfa \in \bbz^2$,
$\bft_{3}$ acts as the identity 
on
$\coker(I-\cals)$. We obtain an induced action of
$\bbzphi/\angles{\bft_{3}} \cong \Phi$ on $\coker(I-\cals)$, and so
$H^1(\Phi;\coker(I-\cals))$ is defined.
	
Suppose we have a class in $H^2(\bbzphi; \bbz^2)$ defining an extension $Q$.
Since $\bbz^2\subset \bbr^2$ has the unique automorphism extension
property, there exists a push-out $\wt Q$
\cite[(5.3.4)]{seif-book}
fitting the commuting diagram
$$
\CD
1 @>>> \bbz^2 @>>> Q      @>>> \bbzphi @>>> 1\\
@. @VVV @VVV @| @.\\
1 @>>> \bbr^2 @>>> \wt Q      @>>> \bbzphi @>>> 1.\\
\endCD
$$
Note that $H^2(\bbzphi; \bbr^2)$ is annihilated by the (finite) index 
of $\bbz=\angles{\bft_{3}}$ in $\bbzphi$ \cite[Proposition 10.1]{Brown}.
Therefore, $H^2(\bbzphi; \bbr^2)$ vanishes, and $\wt Q$ is the split extension $\bbr^2 \rx \bbzphi$.
Since $\bbz\subset\bbzphi$ lifts back to $\gamsig$, it lifts back to
$\wt Q$
so that  $\wt Q$ contains $(\bzero,\bft_{3})\in \bbr^2\rx\bbzphi$.
For each element $\bft_{3}^n \alphabar\in\bbzphi$, pick a preimage
$\alpha= (a, \bft_{3}^n \alphabar)\in \bbr^2\rx\bbzphi$, taking care that
$a=\bzero$ if $\alphabar=\id$.
Then $\bft_{3}^n \alphabar\mapsto a$ defines a map
$\eta: \bbzphi\ra \bbr^2/\bbz^2= T^2$, and in fact, $\eta$ maps into
$\coker(I-\cals) \subset T^2$. Thus we have
\[
\eta: \Phi\ra \coker(I-\cals).
\]
We claim that $\eta$ is a crossed homomorphism.
Let $\alphabar, \betabar\in\Phi$, and
$\eta(\alphabar)=\bfa,\ \eta(\betabar)=\bfb$. For preimages $(\bfa,\bft_{3}^m\alphabar)$ and
$(\bfb,\bft_{3}^n\betabar)$ in $\wt Q$,
\begin{align*}
(\bfa,\bft_{3}^m\alphabar)(\bfb,\bft_{3}^n\betabar)
&=(\bfa+\varphi(\bft_{3}^m\alphabar)(\bfb),\bft_{3}^m\alphabar \bft_{3}^n\betabar)\\
&=(\bfa+\varphi(\bft_{3}^m)(\varphi(\alphabar)(\bfb)),
\bft_{3}^m(\alphabar \bft_{3}^n\alphabar\inv)\alphabar\betabar).
\end{align*}
Since $\alphabar \bft_{3}^n\alphabar\inv=\bft_{3}^\ell$ for some $\ell\in\bbz$,
\begin{align*}
\eta(\alphabar \betabar)
&=\bfa+\varphi(\bft_{3}^m)(\varphi(\alphabar)(\bfb))\\
&=\eta(\alphabar)+\cals^m(\varphi(\alphabar)(\eta(\betabar)))\\
&=\eta(\alphabar)+\varphi(\alphabar)(\eta(\betabar)),
\end{align*}
where the last equality holds because
$\varphi(\alphabar)(\eta(\betabar))\in\coker(I-\cals)$, and
the action of $\cals$ on $\coker(I-\cals)$ is trivial
(if $\bfa\in\coker(I-\cals)$, then
$(I-\cals)\bfa\in\bbz^2$, and hence $\bfa=\cals\bfa$ modulo $\bbz^2$).
Thus $\eta$ is a crossed homomorphism. Conversely, such a crossed homomorphism $\eta$ clearly gives rise to an
extension $Q$. Thus, we obtain a surjective map
\[
Z^1(\Phi; \coker(I-S)) \ra H^2(\bbzphi; \bbz^2),
\]
which we claim is a homomorphism. To see this, given 
\[
\eta: \Phi \ra \coker(I-\cals),
\] we find a $2$-cocycle
$f:\bbzphi \x \bbzphi \ra \bbz^2$ representing the extension $Q$ corresponding to $\eta$. 
Fix a lift $\wt \eta: \Phi \ra (I-S)^{-1}(\bbz^2)$ (not a homomorphism in general) of $\eta$. 
Then we can write any element of $Q$ as 
\[
(\bfn +  \wt \eta(\alphabar), \bft_{3}^{m} \alphabar),
\]
where $\bfn \in \bbz^2$, $m \in \bbz$.
Now, for $(\bfn_1 + \wt \eta(\alphabar), \bft_{3}^{m_1} \alphabar) \text{ and }  (\bfn_2 + \wt \eta(\betabar), 
\bft_{3}^{m_2} \betabar) \in Q,$
\begin{align*}
&(\bfn_1 + \wt \eta(\alphabar), \bft_{3}^{m_1} \alphabar) (\bfn_2 + \wt \eta(\betabar), 
\bft_{3}^{m_2} \betabar)=\\ &(\bfn_1 + \cals^{m_1} \varphi(\alphabar)(\bfn_2) + \wt \eta(\alphabar) +  
\cals^{m_1} \varphi(\alphabar)( \wt \eta(\betabar)), \bft_{3}^{m_1} \alphabar \bft_{3}^{m_2} \betabar ).
\end{align*}
Therefore, $Q$ is represented by the $2$-cocycle $f:\bbzphi \x \bbzphi \ra \bbz^2$ defined by
\begin{align*}
f( \bft_{3}^{m_1} \alphabar, \bft_{3}^{m_2} \betabar  ) = \wt \eta(\alphabar) +  
\cals^{m_1} \varphi(\alphabar)( \wt \eta(\betabar))- \wt \eta(\alphabar \betabar).
\end{align*}
It is now clear that addition of crossed homomorphisms in $Z^1(\Phi; \coker(I-S))$
corresponds to addition of $2$-cocycles in $Z^2(\bbzphi; \bbz^2)$.
	
We shall prove that $Q$ splits if and only if the corresponding $\eta$
is a coboundary, i.e. $\eta\in B^1(\Phi;\coker(I-\cals))$. Note that this will imply that 
$Z^1(\Phi; \coker(I-S)) \ra H^2(\bbzphi; \bbz^2)$ induces an isomorphism
 \[
H^1(\Phi; \coker(I-S)) \cong H^2(\bbzphi; \bbz^2).
\]
A splitting $\bbzphi\ra Q$ induces a homomorphism
\[
s: \bbzphi\ra \wt Q.
\]
Suppose $s(\bft_{3})=(z,\bft_{3})$ with $z\in\bbz^2$.
Even in this case, our definition of $\eta$ shows that, we will pick
$(\bzero,\bft_{3})$ as our preimage of $\bft_{3}$
so that $\eta(\bft_{3})=\bzero$, and
$\eta(\alphabar)=\bfa$ if $s(\alphabar)=(\bfa,\alphabar)$ for
others.

Let
$y=-(I-\cals)\inv z$. Then
\begin{align*}
(y,I) (z,\bft_{3}) (-y,I)&=(y+z-\varphi(\bft_{3})(y), \bft_{3})=(z+(I-\cals)(y), \bft_{3})\\&=(\bzero, \bft_{3})
\end{align*}
and
\begin{align*}
(y,I) (\bfa,\alphabar) (-y,I)&=(y+\bfa-\varphi(\alphabar)y,\alphabar)=(\bfa+(I-\varphi(\alphabar))y,\alphabar)\\
&=(\bfv,\alphabar),  \mbox{ by setting } \bfa+(I-\varphi(\alphabar))y=\bfv.
\end{align*}
Now,
\begin{align*}
(\bfv,\alphabar) (\bzero,\bft_{3})(\bfv,\alphabar)\inv
&=(\bfv-(\alphabar\bft_{3} \alphabar\inv) \bfv,\alphabar\bft_{3} \alphabar\inv)
=(\bfv-\bft_{3}^{\abar} \bfv,\bft_{3}^{\abar})\\
&=((I-\cals^{\abar})\bfv,\bft_{3}^{\abar}).
\end{align*}
Since $\bbz$ is normal in $\bbzphi$, for $s$ to be a homomorphism,
we must have $(I-\cals^{\abar})\bfv=\bzero$. This happens if and only if
$\bfv=0$ since $(I-\cals^{\abar})$ is invertible, which holds if and only if
\[
\eta(\alphabar)=\bfa=(\varphi(\alphabar)-I)(-y)=(\delta y) (\alphabar),
\]
so that $\eta$ is a coboundary.
\end{proof}

An alternate argument for Theorem \ref{H1-cokernel-theorem} is provided by
the long exact sequence
\[ 
\cdots \ra H_{\varphi}^1(\bbzphi; \bbr^2) \ra H_{\varphi}^1(\bbzphi; T^2) \ra H_{\varphi}^2(\bbzphi; \bbz^2) 
\ra H_{\varphi}^2(\bbzphi; \bbr^2) \ra \cdots,  
 \]
induced by the short exact sequence of coefficients $0 \ra \bbz^2 \ra \bbr^2 \ra  T^2 \ra 0$.

Since both $ H_{\varphi}^1(\bbzphi; \bbr^2)$ and $ H_{\varphi}^2(\bbzphi; \bbr^2)$ vanish, we obtain an isomorphism
\[  
H_{\varphi}^1(\bbzphi; T^2) \cong H_{\varphi}^2(\bbzphi; \bbz^2).
\]

To establish that $H_{\varphi}^1(\bbzphi; T^2) \cong
H^1(\Phi;\coker(I-\cals))$, note that any class in $H_{\varphi}^1(\bbzphi;
T^2)$ is represented by a crossed homomorphism, mapping $\bft_3$ to the
identity of $T^2$, and such a crossed homomorphism $\wt \eta: \bbzphi \ra
T^2$ induces $\eta: \Phi \ra T^2$. The image of $\eta$ must lie in $(I-S)
\inv \bbz^2 / \bbz^2$, and so $\eta$ defines an element of
$H^1(\Phi;\coker(I-\cals))$.
It is straightforward to check that this is an isomorphism.

On the other hand, our proof of Theorem \ref{H1-cokernel-theorem}
establishes the precise one-one correspondence between
$H^1(\Phi;\coker(I-\cals))$ and the set of all equivalence classes of
extensions $Q$,
\[
1\lra \bbz^2\lra Q\lra \bbzphi\lra 1.
\]

\begin{remark}
\label{cocycleConditions}
For each subgroup $\Phi$ of $D_4$, we describe both
$Z^1(\Phi;\coker(I-\cals))$ and $B^1(\Phi;\coker(I-\cals))$, where the
action of $\Phi$ on $\coker(I-\cals)$ is induced from a $\varphi:\bbzphi
\ra \gltz$ in Theorem \ref{abstract-kernel}.  For $\Phi \cong \bbz_2 \x
\bbz_2$, we need to check that the commutator of $(\bfa, \alphabar)$ and
$(\bfb, \betabar)$ is in $\bbz^2$. For $\bbz_4$, there is no cocycle
condition to check (since $I+ \varphi(\alphabar) + \varphi(\alphabar)^2 +
\varphi(\alphabar) ^3 = 0$). Likewise for $\bbz_4 \rx \bbz_2$, there is no
cocycle condition for the order 4 element.
\end{remark}

\vskip4pt
\noindent
(1)
$\Phi=\bbz_2=\angles{\alphabar}$,
\begin{align*}
Z^1(\Phi;\coker(I-\cals))
&=\{\bfa\in\coker(I-\cals)\mid\ 
(I+\varphi(\alphabar))\bfa\equiv\bzero\}\\
B^1(\Phi;\coker(I-\cals))
&=\{(I-\varphi(\alphabar))\bfv\mid\ \bfv\in\coker(I-\cals)\}
\end{align*}
\vskip4pt
\noindent
(2)
$\Phi=\bbz_4=\angles{\alphabar}$,
\begin{align*}
Z^1(\Phi;\coker(I-\cals))
&=\{\bfa\in\coker(I-\cals)\}\\
B^1(\Phi;\coker(I-\cals))
&=\{(I-\varphi(\alphabar))\bfv\mid\ \bfv\in\coker(I-\cals)\}
\end{align*}
\vskip4pt
\noindent
(3)
$\Phi=\bbz_2\x\bbz_2=\angles{\alphabar,\betabar}$,
\begin{align*}
Z^1(\Phi;\coker(I-\cals))
&=\{(\bfa,\bfb) \mid \bfa,\bfb \in \coker(I-\cals),\\
&\phantom{AAA}(I+\varphi(\alphabar))\bfa\equiv (I+\varphi(\betabar))\bfb\equiv\bzero, \\
&\phantom{AAA}(I-\varphi(\alphabar))\bfb\equiv(I-\varphi(\betabar))\bfa\}\\
B^1(\Phi;\coker(I-\cals))
&=\{((I-\varphi(\alphabar))\bfv,(I-\varphi(\betabar))\bfv) \mid
\bfv\in\coker(I-\cals)\}
\end{align*}
\vskip4pt
\noindent
(4)
$\Phi=\bbz_4\rx\bbz_2=\angles{\alphabar,\betabar |\
\alphabar^2, \betabar^2, (\betabar\alphabar)^4}$,
\begin{align*}
Z^1(\Phi;\coker(I-\cals))
&=\{(\bfa,\bfb) \mid \bfa,\bfb \in \coker(I-\cals),\\ 
&\phantom{AAA}(I+\varphi(\alphabar))\bfa\equiv(I+\varphi(\betabar))\bfb\equiv\bzero\}\\
B^1(\Phi;\coker(I-\cals))
&=\{((I-\varphi(\alphabar))\bfv,(I-\varphi(\betabar))\bfv) \mid
\bfv\in\coker(I-\cals)\}
\end{align*}
\bigskip

Suppose we have
an extension $Q$; that is, $\eta\in H^1(\Phi;\coker(I-\cals))$
with $\eta(\alphabar)=\bfa=\xy{a_1}{a_2}$. Then
\[
Q=\angles{\bft_1,\bft_2,\bft_3,
\alpha=(\bft_1^{a_1} \bft_2^{a_2} \bft_3^{a_3},A)} \subset \Sol \rx D_4
\]
has the following presentation
\begin{align*}
\hspace{0.5in}
&\bft_3(\bft_1^{n_1}\bft_2^{n_2})\bft_3\inv=\bft_1^{m_1}\bft_2^{m_2},
\text{ where } {\xy{m_1}{m_2}=\cals\xy{n_1}{n_2}},\\
&\alpha(\bft_1^{n_1}\bft_2^{n_2})\alpha\inv=\bft_1^{m_1}\bft_2^{m_2},
\text{ where } {\xy{m'_1}{m'_2}=\varphi(\alphabar)\xy{n_1}{n_2}},\\
&\alpha\bft_3\alpha\inv=\bft_1^{w_1}\bft_2^{w_2}\bft_3^{\abar},
\text{ where } {\xy{w_1}{w_2}=
\left(I-\cals^{\abar}\right)\xy{a_1}{a_2}},\\
&\alpha^2=\bft_1^{v_1}\bft_2^{v_2}\bft_3^{(1+\abar)a_3},
\text{ where } {\xy{v_1}{v_2}=\left(I+\varphi(\alphabar)\right)\xy{a_1}{a_2}},
\text{ if }\ A^2=I,\\
&\alpha^4=\id,
\text{ if }\ \ord(A)=4.
\end{align*}

\begin{corollary}
\label{a-conj-zero3}
Let $Q = \angles{\Gamma_\cals,
(\bft_1^{a_1}\bft_2^{a_2}\bft_3^{a_3},A)}$ be a
$\Sol$-crystallographic group with standard lattice $\Gamma_\cals =
\angles{\bft_1,
\bft_2, \bft_3}$. Suppose $\varphi(\alphabar)=-K$ and $\cals=nK\pm I$.
Recall that by Theorem \ref{abstract-kernel},
$A$ has order 2, $\abar=1$, and $a_3=\frac 12$.
Then 
\[
H^1(\Phi;\coker(I-\cals))=0.
\]
In fact, there exists $\bft_1^{v_1}\bft_2^{v_2}$ which conjugates
$(\bft_1^{a_1}\bft_2^{a_2}\bft_3^{\frac12},A)$ to $(\bft_3^{\frac12},A)$
and leaves $\Gamma_\cals$ invariant.
\end{corollary}

\begin{proof}
We have $\det(I - \varphi(\alphabar)) = \det(I + K) = 1 + \det(K) +\tr(K)$. By Theorem \ref{abstract-kernel}, 
when $\det(K)=-1$, $\tr(K)>0$; and when $\det(K)=1$,  $\tr(K) > 2$.

Consequently, $I - \varphi(\alphabar)$ is always non-singular and we may take
$\bfv=(I-\varphi(\alphabar))\inv\bfa$.  Then $(\bft_1^{v_1} \bft_2^{v_2},I) \in \Sol \rx D_4$ conjugates
$(\bft_3^{\frac 1 2}, A)$ to $(\bft_1^{a_1} \bft_2^{a_2}\bft_3^{\frac 1 2},
A)$.
It remains to show $\bfv\in (I-\cals)\inv\bbz^2$. This condition guarantees
conjugation by $(\bft_1^{v_1} \bft_2^{v_2},I)$ leaves $\gamsig$  invariant.
Since $\varphi(\alphabar)=-K$ is a square root of $\cals$ and $\bfv=(I-\varphi(\alphabar))\inv\bfa,$
\[
(I-\cals) \bfv = (I+\varphi(\alphabar))(I-\varphi(\alphabar)) \bfv =
(I+\varphi(\alphabar))\bfa \in \bbz^2,
\]
where the last inclusion holds by the cocycle conditions in Remark \ref{cocycleConditions}. 
Therefore $\Phi \ni A \mapsto \bfa \in \coker(I-\cals)$ is a coboundary, 
and $H^1(\Phi;\coker(I-\cals))$ vanishes.
\qedhere
\end{proof}

Corollary \ref{a-conj-zero3} greatly simplifies the computation of
$H^1(\Phi; \coker(I-\cals))$. For example, in cases \gr{1}, \gr{2b}, and
\gr{5} of Theorem \ref{abstract-kernel}, we can take $\bfa=\bfzero$,
whereas in cases \gr{6b}, \gr{6b\ii}, \gr{7}, and \gr{7\ii}, we can take
$\bfb=\bfzero$.

The complete list of crystallographic groups for $\Sol$ will follow from
our classification of crystallographic groups of $\Solof$. However, we will
need to analyze how a type \gr{3\ii} or \gr{6\ii} crystallographic group of
$\Sol$ acts on $\Sol$. This will be critical to determining when a
crystallographic group of $\Solof$ has torsion.

\begin{lemma}
\label{geometricLemma}
Let $Q$ be a crystallographic group of\ $\Sol$ of type \gr{3\ii} or \gr{6b\ii}.

When $Q$ is of type \gr{3\ii}, 
\[
Q=\left\langle \gamsig,
\alpha=\left(\bft_1^{a_1} \bft_2^{a_2} \bft_3^{0}, 
\matNpp\right)\right\rangle, \mbox { and}
\] 
$Q \backslash \Sol$ can be described as $T^2 \x I$ with
$T^2 \x \{0 \}$ identified to itself by the affine involution $\left( \xy{a_1} {a_2}, \matPpm \right)$, 
and $T^2 \x \{1 \}$ identified to itself by the affine involution 
$\left( \xy{a_1} {a_2}, \left[ \begin{matrix} \sigma_{11} & -\sigma_{12}\\
\sigma_{21} & -\sigma_{11} \end{matrix} \right] \right)$. Here $T^2$ is the 2-dimensional torus.

If $\left[\begin{matrix} -1 & 0 \\ 0 & 1\end{matrix} \right]$ is used instead of $\left[\begin{matrix} 1 & 0 \\ 0 & -1\end{matrix} \right]$,
then  $Q \backslash \Sol$ can be described as $T^2 \x I$ with
$T^2 \x \{0 \}$ identified to itself by the affine involution $\left( \xy{a_1} {a_2},
\left[\begin{matrix} -1 & 0 \\ 0 & 1\end{matrix} \right] \right)$, and
 $T^2 \x \{1 \}$ identified to itself by the affine involution
$\left( \xy{a_1} {a_2}, \left[ \begin{matrix} -\sigma_{11} & \sigma_{12}\\
-\sigma_{21} & \sigma_{11} \end{matrix} \right] \right)$.

When $Q$ is of type \gr{6b\ii},
 \[
Q=\left\langle \gamsig,
\alpha=\left(\bft_1^{a_1} \bft_2^{a_2} \bft_3^{0}, \matNpp\right),\beta=\left(\bft_3^{\frac12}, \matPmm\right)\right\rangle
, \mbox{ and}
\] 
$Q \backslash \Sol$ can be described as $T^2 \x I$ with
$T^2 \x \{0 \}$ identified to itself by the affine involution $\left( \xy{a_1} {a_2}, \matPpm \right)$, and $T^2 \x \{1 \}$
 identified to itself by the affine involution 
$\left( \xy{a_1} {a_2}, \left[ \begin{matrix} -k_{11} & k_{12}\\
-k_{21} & k_{11} \end{matrix} \right] \right)$.
\end{lemma}
\begin{proof}
The action of
$\Gamma_\cals$ on $\Sol$ is equivalent to the action of $\bbz^2 \rx_{\cals}
\bbz$ on $\bbr^2 \rx_\cals \bbr$.  A fundamental domain for this action is
given by the unit cube $I^3$, and evidently $Q \backslash \Sol$ is
given by $T^2 \x I$ with $T^2 \x \{0\}$ identified to $T^2 \x \{1\}$ via
$\cals$, which we view as a self-diffeomorphism of $T^2$.  Note
that $\bbr^2 \ra \bbr^2 \rx_\cals \bbr \ra \bbr$ induces the fiber bundle
with infinite cyclic structure group generated by $\cals$:
\[
T^2 \ra {\Gamma_\cals} \backslash \Sol \ra S^1.
\]

Now suppose $Q$ is of type \gr{3\ii}. Then $Q\bs\Sol$ is the quotient of
${\Gamma_\cals} \backslash \Sol$ by the involution defined by $\alpha =
\left(\bft_1^{a_1} \bft_2^{a_2}, \matNpp\right)$. Here $\alpha$ acts as a
reflection on the base $S^1$. A fundamental domain for this action is given
by $T^2 \x \left[0, \frac 1 2 \right]$. Now $\alpha$ identifies $T^2 \x
\{0\}$ to itself and $T^2 \x \{\frac 1 2\}$ to itself.

Indeed, $(\bft_1^{a_1}\bft_2^{a_2},A) \cdot \bft_1^{x_1} \bft_2^{x_2} =
\bft_1^{a_1}\bft_2^{a_2}A(\bft_1^{x_1} \bft_2^{x_2})$ shows that $\alpha$
acts on $T^2 \x \{0\}$ as the affine transformation
$(\bfa,\varphi(\alphabar))$.
For $\bft_1^{x_1} \bft_2^{x_2} \bft_3^{\frac 1 2} \in T^2 \x \{\frac 1 2\}$,
\begin{align*}
\bft_3 (\bft_1^{a_1}\bft_2^{a_2},A) \cdot \bft_1^{x_1} \bft_2^{x_2}\bft_3^{\frac 1 2}& 
=\bft_3 \bft_1^{a_1}\bft_2^{a_2} A(\bft_1^{x_1} \bft_2^{x_2}) A(\bft_3^{\frac 1 2})
=\bft_3 \bft_1^{a_1}\bft_2^{a_2} A(\bft_1^{x_1} \bft_2^{x_2}) \bft_3^{-\frac 1 2}\\
&=\left( \bft_3 \bft_1^{a_1}\bft_2^{a_2} \bft_3^{-1} \right)
\left( \bft_3 A(\bft_1^{x_1} \bft_2^{x_2}) \bft_3^{-1} \right) \bft_3^{\frac 1 2} \in T^2 \x \{\tfrac 1 2\}.
\end{align*}
Since conjugation by $\bft_3$ is the action of $\cals$, we see that
$\alpha$ acts on $T^2$ as the affine transformation $(\cals \bfa, \cals
\varphi(\alphabar))$. But since $\bfa \in \coker(I-\cals)$, this simplifies
to $(\bfa, \cals \varphi(\alphabar))$.  Note that the condition that
$\sigma_{11}=\sigma_{22}$ ensures that $\cals \varphi(\alphabar)$ has order
2.

The argument in case \gr{6b\ii} is nearly identical. In this case, note that
$Q$ contains a group of type \gr{2b}, say $Q'$, as an index 2 subgroup,
\[
Q'=\left\langle \gamsig,
\beta=\left(\bft_3^{\frac12}, \matPmm\right)\right\rangle.
\] 

Therefore, $Q \backslash \Sol$ is the quotient of $Q' \backslash \Sol$ by
$\alpha = \left(\bft_1^{a_1} \bft_2^{a_2}, \matNpp\right)$. Now $Q'
\backslash \Sol$ is the quotient of ${\Gamma_\cals} \backslash \Sol$ by the
involution defined by $\beta$. On the base of $T^2 \ra {\Gamma_\cals}
\backslash \Sol \ra S^1$, $\beta$ acts as a translation. Thus a fundamental
domain for the action of $\beta$ is given by $T^2 \x \left[0, \frac 1 2
\right]$. Note that $\beta$ identifies $T^2 \x \{0\}$ with $T^2 \x
\{\frac 1 2\}$ via $\varphi(\betabar) = -K$, which is a square root of
$\cals$, and $Q' \backslash \Sol$ is the mapping torus of
$\varphi(\betabar)$. Now because $Q' \backslash \Sol$ admits the structure
of a $T^2$ bundle over $S^1$, the construction in \gr{3\ii} applies. A
fundamental domain for the action of $\alpha$ on $Q' \backslash \Sol$ is
given by $T^2 \x \{\frac 1 4\}$. As in case \gr{3\ii}, $\alpha$ acts on
$T^2 \x \{0\}$ affinely as $(\bfa,\varphi(\alphabar))$. For $\bft_1^{x_1}
\bft_2^{x_2} \bft_3^{\frac 1 4} \in T^2 \x \{\frac 1 4\}$,

\begin{align*}
(\bft_3^{\frac12}, B) (\bft_1^{a_1}\bft_2^{a_2},A) \cdot &\bft_1^{x_1} \bft_2^{x_2}\bft_3^{\frac 1 4} =
\bft_3^{\frac12} B(\bft_1^{a_1}\bft_2^{a_2}) BA(\bft_1^{x_1} \bft_2^{x_2}) BA(\bft_3^{\frac 1 4})\\
&=\bft_3^{\frac12} B(\bft_1^{a_1}\bft_2^{a_2})
\bft_3^{-\frac12} \bft_3^{\frac12} BA(\bft_1^{x_1} \bft_2^{x_2}) \bft_3^{-\frac 1 4}\\
&=\left(\bft_3^{\frac12} B(\bft_1^{a_1}\bft_2^{a_2})
\bft_3^{-\frac12} \right) \left( \bft_3^{\frac12} BA(\bft_1^{x_1} \bft_2^{x_2})  \bft_3^{-\frac 1 2} \right) \bft_3^{\frac 1 4}
\in T^2 \x \{\tfrac 1 4\}.
\end{align*}

Now conjugation by $(\bft_3^{\frac 12}, B)$ is the action of
$\varphi(\betabar)=-K$ on $T^2$. Hence $\alpha$ acts affinely on $T^2 \x
\{\frac 1 4\}$ as $(\varphi(\betabar)\bfa, \varphi(\betabar)
\varphi(\alphabar))$.  The commutator cocycle conditions for $\Phi=\bbz_2
\times \bbz_2$ in Remark \ref{cocycleConditions}, with $\bfb=\bfzero$
implies $(I-\varphi(\betabar))\bfa = (I+K)\bfa \in \bbz^2$, so this
simplifies to $(\bfa, \varphi(\betabar) \varphi(\alphabar))=(\bfa, (-K)
\varphi(\alphabar))$.
\qedhere
\end{proof}

\section
{\texorpdfstring
{Lattices of $\Solof$}
{Lattices of Solof}}

In this section we classify the lattices of $\Solof$.
Given a lattice $\tgamsig$ of $\Solof$, $\wt\Gamma \cap \calz(\Solof) \cong \bbz$ is a lattice of $\calz(\Solof) \cong \bbr$, and
the projection map, 
\[
G \ra G / \calz(G) \cong \Sol,
\] 
carries $\tgamsig$ to a lattice of $\Sol$, 
isomorphic to $\Gamma_{\cals}$, for some $\cals\in\sltz$ with $\rm{trace}(\cals)>2$. Thus, $\tgamsig$ is the central extension
\[ 
1 \lra \bbz \lra \tgamsig \lra \Gamma_{\cals} \lra 1.
\]
As is well known, such central extensions of $\bbz$ by $\gamsig$ are
classified by the second cohomology group $H^2(\gamsig;\bbz)$.

\begin{theorem}
\label{lattice-part-construction}
Let $\cals\in\sltz$ with $\rm{trace}(\cals)>2$.
There is a one-one correspondence between the equivalence classes of all central extensions
\[
1\lra \bbz\lra \wt\Gamma\ \lra\gamsig\lra 1
\]
and the group $\bbz \oplus \text{Coker}(\cals-I)$.  Note $\coker(\cals-I)$ is finite.
\end{theorem}

\begin{proof}
Recall $\gamsig=\bbz^2 \rx_\cals \bbz$. Then
\begin{align*}
H^2(\bbz^2\rx_{\cals}\bbz;\bbz)
&=\text{Free} \left(H_2(\bbz^2\rx_{\cals}\bbz;\bbz)\right)\oplus
\text{Torsion} \left(H_1(\bbz^2\rx_{\cals}\bbz;\bbz)\right)\\
&=\bbz \oplus (\bbz^2/(\cals-I)\bbz^2)
=\bbz \oplus \text{Coker}(\cals-I).
\end{align*}
\qedhere
\end{proof}

For $\{q,(m_1,m_2)\}\in \bbz \oplus \text{Coker}(\cals-I)$,
denote the corresponding extension $\wt\Gamma$ by $\lgamsig$
whose presentation is given in Lemma 	\ref{determine-TT}.
We show that $\lgamsig$ with $q \neq 0$ embeds as a lattice in $\Solof$
(when $q=0$, $\lgamsig$ embeds into $\Sol \x \bbr$).
An $\cals\in\sltz$ with $\tr(\cals)>2$ produces $P$ and $\Delta$,
where $P\in \sltr$ diagonalizes $\cals$, $P \cals P \inv
=\left[
\begin{matrix} \tfrac{1}{\lambda}&0\\
0&\lambda \end{matrix}
\right],
$ $\tfrac{1}{\lambda}<1<\lambda$. We had the embedding of $\bbz^2\rx_\cals\bbz$ into $\Sol$ in (\ref{def-phi}):
\begin{align*}
\notag \left(\xy{x}{y},u\right)
\longmapsto
\notag \left(P\xy{x}{y},u\ln(\lambda)\right).
\end{align*}
The quotient of $\Solof$ by its center is isomorphic to $\Sol$ by the projection
$$
\left[
\begin{matrix} 1 &e^u x & z\\
0 &e^u &y\\
0 &0 &1 \end{matrix}
\right]
\longmapsto
\left(\left[\begin{matrix} x\\ y \end{matrix}\right],u\right).
$$
Under this projection, we will find all lattices of $\Solof$ projecting to $\gamsig$. Let
\begin{align}
\label{embedSolofLattice}
\begin{split}
\bfe_1&=\left(\left[\begin{matrix} 1\\ 0 \end{matrix} \right],0\right)
\longmapsto
(P\bfe_1,0)\longmapsto
\bft_1=
\left[\begin{matrix}
1 & p_{11} &c_1\\
0 & 1 &p_{21}\\
0 & 0 &1 \end{matrix}
\right]
,\\
\bfe_2&=\left(\left[\begin{matrix} 0\\ 1 \end{matrix} \right],0\right)
\longmapsto
(P\bfe_2,0)\longmapsto
\bft_2=
\left[\begin{matrix}
1 & p_{12} &c_2\\
0 & 1 &p_{22}\\
0 & 0 &1 \end{matrix}
\right],\\
\bfe_3&
=\left(\left[\begin{matrix} 0\\ 0 \end{matrix} \right],1\right)
\longmapsto
(0,\ln(\lambda))\longmapsto
\bft_3=
\left[\begin{matrix}
1 & 0 &c_3\\
0 & {\lambda} &0\\
0 & 0 &1 \end{matrix}
\right],\\
&\kern6pt\phantom{=\left(\left[\begin{matrix} 0\\
1 \end{matrix} \right],0\right)
\longmapsto}
\bft_4=
\left[\begin{matrix}
1 & 0&1\\
0 & 1 &0\\
0 & 0 &1 \end{matrix}
\right].
\end{split}
\end{align}
where $c_i$'s are to be determined. Then $[\bft_1,\bft_2]=\bft_4$
(regardless of the $c_i$'s).

\begin{lemma}
\label{determine-TT}
For any integers $q, m_1,m_2$,
there exist unique $c_1, c_2$ for which $\{\bft_1,\bft_2,\bft_3,\bft^{\frac1q}_4\}$
forms a group $\lgamsig$ with the presentation
\begin{align*}
\lgamsig=\langle
\bft_1,\bft_2,\bft_3,\bft_4^{\frac 1q}\ |\
[\bft_1,\bft_2]&=\bft_4,\ \text{$\bft_4$ is central,}\\
\bft_3 \bft_1 \bft_3\inv
&= \bft_1^{\sigma_{11}} \bft_2^{\sigma_{21}} \bft_4^{\frac{m_1}{q}},\\
\bft_3 \bft_2 \bft_3\inv&= \bft_1^{\sigma_{12}} \bft_2^{\sigma_{22}}
\bft_4^{\frac{m_2}{q}}\rangle.
\end{align*}
	
Consequently, $\lgamsig$ is solvable and contains $\Gamma_q =\angles{\bft_1, \bft_2, \bft^{\frac 1q}_4}$ 
as its discrete nil-radical, where $\Gamma_q$ is a lattice of $\nil$.
\end{lemma}

\begin{proof}
We only need to verify the last two equalities. But they
become a system of  equations on $c_i$'s
\begin{align}
\label{c1c2-equation}
\begin{split}
(1-\sigma_{11})c_1 -\sigma_{21} c_2  &={\frac{m_1}{q}}
-\frac{\sigma_{21} (\sigma_{12}+1 - \sigma_{11} + \sigma_{11} \sqrt{T^2-4})}{2 \sqrt{T^2-4}},\\
-\sigma_{12} c_1+ (1-\sigma_{22})c_2 &={\frac{m_2}{q}}
+\frac{\sigma_{12} (\sigma_{21} +1 - \sigma_{22} - \sigma_{22} \sqrt{T^2-4})}{2 \sqrt{T^2-4}},
\end{split}
\end{align}
where $T=\sigma_{11}+\sigma_{22}$.
Since $I-\cals$ is non-singular, there exists a unique solution for
$c_1, c_2$.
\end{proof}

Equation (\ref{c1c2-equation}) also shows the cohomology
classification.
Suppose $\{c_1,c_2\}$ and $\{c'_1,c'_2\}$ are solutions for the equations
with $\{m_1,m_2\}$ and $\{m'_1,m'_2\}$, respectively.
Then $(c'_1-c_1,c'_2-c_2)\in\big(\frac1q\bbz\big)^2$ if and only if
$(m'_1-m_1,m'_2-m_2)\in \text{Coker}(\cals^T-I) \cong \coker(\cals-I)$.
This happens if and only if
$\lgamsig=\wt\Gamma_{(\cals;q,m'_1,m'_2)}$.

\begin{remark}
(1) Note that any lattice $\lgamsig$ of $\Solof$ projects to the standard lattice
$\gamsig$ of $\Sol$.

(2) In Lemma \ref{determine-TT}, the $c_i$'s are independent of choice of $P$ because
equation (\ref{c1c2-equation}) has coefficients only from the matrix $\cals$.
	
(3) Notice that $c_3$ does not show up in the presentation of the lattice
$\lgamsig$, so $c_3$ can be changed without affecting the
isomorphism type of the lattice.
\end{remark}

\begin{notation}[Standard lattice]
\label{standard-lattice}
The lattice generated by
\[
\bft_1=
\left[\begin{matrix}
1 & p_{11} &c_1\\
0 & 1 &p_{21}\\
0 & 0 &1 \end{matrix}
\right]
,\\
\bft_2=
\left[\begin{matrix}
1 & p_{12} &c_2\\
0 & 1 &p_{22}\\
0 & 0 &1 \end{matrix}
\right],\\
\bft_3=
\left[\begin{matrix}
1 & 0 &c_3\\
0 & {\lambda} &0\\
0 & 0 &1 \end{matrix}
\right],\\
\bft_4^{\frac 1q}=
\left[\begin{matrix}
1 & 0&\frac 1q\\
0 & 1 &0\\
0 & 0 &1 \end{matrix}
\right]
\]
{\bf with $c_3= 0$},
is called a \emph{standard lattice} of $\Solof$.
\end{notation}
Therefore, any lattice of $\Solof$ is isomorphic to a standard lattice.
However, a non-standard lattice (i.e., $c_3\not=0$) will be needed when we consider 
finite extensions of $\wt \Gamma_\cals$, specifically, in the holonomy $\bbz_4$ case.

\noindent
The following  lemma on lattices of $\Solof$ will be needed in the next section.

\begin{lemma}
\label{rational}

Let $\lgamsig$ be a lattice of $\Solof$, embedded as in assignment (\ref{embedSolofLattice}).
\noindent
\rm{(a)} Let $r_1, r_2 \in \bbq$. Then
\[
\bft_1^{r_1} \bft_2^{r_2} = \bft_2^{r_2} \bft_1^{r_1} \bft_4^{r_1 r_2}.
\]

\noindent
\rm{(b)}
Let $a_1, a_2 \in \bbq$. Then, for $\abar=\pm 1$,
\begin{align}
\label{thev}
\bft_3^{\abar} \bft_1^{a_1} \bft_2^{a_2} \bft_3^{-\abar} = \bft_1^{l_1} \bft_2^{l_2} \bft_4^{v}, 
\text{ where } 
{\xy{l_1}{l_2}=
\cals^{\abar}\xy{a_1}{a_2}}, \text{ and } v \in \bbq.
\end{align} 
\end{lemma}

\begin{proof}
For part (a), we compute that
$ \left[ \bft_1^{r_1}, \bft_2^{r_2} \right] = \bft_4^{r_1 r_2 \det(P)} = \bft_4^{r_1 r_2}.$  

For part (b), the definition of $\lgamsig$ shows that ${\xy{l_1}{l_2}=
\cals^{\abar}\xy{a_1}{a_2}}$. 
We must show that $v$ in (\ref{thev}) is rational.

Because $a_1$ and $a_2$ are rational, there is a positive integer $n$
so that $n a_1, n a_2 \in \bbz$. By part (a),
\begin{align*}
(\bft_1^{a_1} \bft_2^{a_2})^n = \bft_1^{a_1} \bft_2^{a_2} \cdots \bft_1^{a_1} \bft_2^{a_2} (n \text{ times})
= \bft_1^{n a_1} \bft_2^{n a_2} \bft_4^{u'}, \text{ for some } u' \in \bbq.
\end{align*}
Therefore, 
\begin{align*}
\bft_3^{\abar} (\bft_1^{a_1} \bft_2^{a_2})^n \bft_3^{-\abar}  
&= \bft_3^{\abar} \bft_1^{n a_1} \bft_2^{n a_2} \bft_4^{u'}  \bft_3^{-\abar}
=\bft_3^{\abar} \bft_1^{n a_1} \bft_2^{n a_2}   \bft_3^{-\abar} \bft_4^{u'}\\
&= \bft_1^{n_1} \bft_2^{n_2} \bft_4^{u} \text{ for some } n_1,n_2 \in \bbz, \text{ and some } u \in \bbq,
\end{align*}
where the last equality follows from that $n a_1$  and $n a_2$ are integers, 
together with the relations in Lemma \ref{determine-TT}.
	
On the other hand, we have that
\begin{align*}
\bft_3^{\abar} (\bft_1^{a_1} \bft_2^{a_2})^n \bft_3^{-\abar}  &= 
\bft_3^{\abar} \bft_1^{a_1} \bft_2^{a_2}\bft_3^{-\abar} \cdots \bft_3^{\abar} \bft_1^{a_1} \bft_2^{a_2}\bft_3^{-\abar} (n \text{ times})\\
&= \bft_1^{l_1} \bft_2^{l_2} \bft_4^{v} \cdots \bft_1^{l_1} \bft_2^{l_2} \bft_4^{v}  (n \text{ times}) \\
&=  \bft_1^{n l_1} \bft_2^{n l_2}  \bft_4^{nv + w} \text{ for some } w \in \bbq,
\end{align*}
where $v$ is from (\ref{thev}) and the last equality follows from part (a).
	
	\noindent
Consequently, we have 
\[
 \bft_1^{n_1} \bft_2^{n_2} \bft_4^{u} = \bft_1^{n l_1} \bft_2^{n l_2}  \bft_4^{nv + w}.
\]
This forces $n_1= n l_1$ and $n_2= {n l_2}$. Therefore, 
$nv + w = u.$
Since $n \in \bbz$, $u, w \in \bbq$, it follows that $v \in \bbq$.

\qedhere
\end{proof}

\section
{\texorpdfstring
{Crystallographic groups of $\Solof$}
{Crystallographic groups of Solof}}
\label{solof-cryst}

Let $\varPi \subset \Solof \rx C$ be a crystallographic group of $\Solof$,
where $C$ is a maximal compact subgroup of $\aut(\Solof)$. As all maximal
compact subgroups of $\Solof$ are conjugate, we can assume that $C$ is the
maximal compact subgroup 
\[
D_4 = \left\langle \left[\begin{matrix} 0 & -1
\\ 1 & 0
\end{matrix} \right], \left[\begin{matrix} 1 & 0 \\ 0 & -1 
\end{matrix} \right] \right\rangle
\] 
of $\aut(\Solof)$ 
(Proposition \ref{BiebSol}), 
the action of which on $\Solof$ is described in Proposition
\ref{autosol14}. As noted in Proposition \ref{BiebSol}, $\Solof$ satisfies
generalization of Bieberbach's Theorems.  Furthermore, as shown below, we
can conjugate $\varPi$ in $\aff(\Solof)$ so that the lattice inside
$\varPi$ is some $\lgamsig$, embedded in $\Solof$ as in assignment
(\ref{embedSolofLattice}).

\begin{proposition}
\label{solofShape}
\noindent
{\rm{(1)}} Any crystallographic group $\varPi'$ of $\Solof$ can be conjugated in $\aff(\Solof)$ to $\varPi \subset \Solof \rx D_4$ so that 
\[
\varPi \cap \Solof=\angles{\bft_1,\bft_2,\bft_3,\bft_4^{\frac1q}},
\]
where
\[
\bft_1=
\left[\begin{matrix}
1 & p_{11} &c_1\\
0 & 1 &p_{21}\\
0 & 0 &1 \end{matrix}
\right],\
\bft_2=
\left[\begin{matrix}
1 & p_{12} &c_2\\
0 & 1 &p_{22}\\
0 & 0 &1 \end{matrix}
\right],\
\bft_3=
\left[\begin{matrix}
1 & 0 &c_3\\
0 & \lambda &0\\
0 & 0 &1 \end{matrix}
\right],\
\bft_4=
\left[\begin{matrix}
1 & 0 &1\\
0 & 1 &0\\
0 & 0 &1 \end{matrix}
\right].
\]
	
\noindent
{\rm{(2)}} The holonomy group $\Phi$ is generated by at most two elements of $D_4$, 
and thus $\varPi$ is generated by $\angles{\bft_1, \bft_2, \bft_3, \bft_4^{\frac1q}}$ and at most two 
isometries of the form $(\bft_1^{a_1} \bft_2^{a_2} \bft_3^{a_3} \bft_4^{a_4} , A)$, for $A \in D_4$ and real numbers $a_i$.
\end{proposition}

\begin{proof}
Let $\wt\Gamma=\varPi\cap\Solof$. This lattice must meet the center of $\Solof$ in a lattice:
$\wt\Gamma\cap \calz(\Solof)$ is a lattice of
$\calz(\Solof)$, say generated by $\bft_4^{\frac1q}$. Also
$\wt\Gamma\cap\Nil$ is a lattice of the nilradical $\Nil$, so we can find
generators $\angles{\bft_1,\bft_2,\bft_4^{\frac1q}}$ of this lattice as
given in the statement.
The remaining one generator for the lattice $\wt\Gamma$ must project
down to a generator of the quotient $\wt\Gamma / \angles{\bft_1,\bft_2,\bft_4^{\frac1q}} \cong \bbz$. It must be of the form
\[
\bft''_3=
\left[\begin{matrix}
1 & a &c_3\\
0 & {\lambda} &b\\
0 & 0 &1 \end{matrix}
\right]
\]
Conjugation by
$
\left[\begin{matrix}
1 & \frac{a}{1-\lambda} & 0\\
0 & {\lambda} &-\frac{b \lambda}{1-\lambda}\\
0 & 0 &1 \end{matrix}
\right]
$
maps $\bft''_3$ to the form of $\bft_3$. Note $\wt\Gamma/\calz(\wt\Gamma)$ is a lattice of $\Sol$, 
isomorphic to $\bbz^2 \rx_\cals \bbz$, for $\cals \in \sltz$, $\tr(\cals)>2$, where $P = (p_{i j})$ diagonalizes $\cals$.
As in the case of $\Sol$ (Proposition \ref{lattice-a=0}), we can assume $\det(P)=1$, so that
$[\bft_1,\bft_2]=\bft_4$.
Therefore, any lattice is conjugate to a lattice
$\angles{\bft_1,\bft_2,\bft_3,\bft_4^{\frac1q}}$ of the desired form.
\qedhere
\end{proof}

Henceforth we will assume all $\Solof$-crystallographic groups are embedded
in $\Solof \rx D_4$ as in Proposition \ref{solofShape}. However, we will
see that we can always take $c_3=0$, except possibly when the holonomy of
$\varPi$, $\Phi$, is $\bbz_4$.  Because lattices of $\Solof$ project to
lattices of $\Sol$, the projections $\Solof \ra \Sol$ and $\aut(\Solof) \ra
\aut(\Sol)$ induce a projection $\Solof \rx D_4 \ra \Sol \rx D_4$ which
carries a $\Solof$-crystallographic group $\varPi$ to a
$\Sol$-crystallographic group $Q$.  Furthermore, when $\varPi$ is embedded
in $\Solof \rx D_4$ as in Proposition \ref{solofShape}, the lattice $\wt
\Gamma_\cals = \lgamsig$ projects to a standard lattice $\gamsig$ of
$\Sol$. That is, we have the following commuting diagram:

\[
\CD
@. 1 @. 1 @.  @.\\
@. @VVV @VVV @. @.\\
@. \frac1q\bbz =\angles{\bft^\frac1q_4} @=  \frac1q\bbz =\angles{\bft^\frac1q_4}@. @.\\
@. @VVV @VVV @. @.\\
1 @>>> \tgamsig @>>> \varPi @>>> \Phi @>>> 1\\
@. @VVV @VVV @| @.\\
1 @>>> \gamsig @>>> Q @>>> \Phi @>>> 1\\
@. @VVV @VVV @. @.\\
@. 1 @. 1 @.  @.\\
\endCD
\]

Our goal is finding all crystallographic groups $\varPi$ of $\Solof$ which
project down to $Q$. In general, it is \emph{not} true that there exists
$\varPi$ fitting the above commutative diagram of exact sequences 
without making the kernel $\angles{\bft_4}$ finer to
$\angles{\bft_4^{1/q}}$. That is, even though $\wt\Gamma_{\cals}$
always exists, for $\varPi$ to exist, sometimes the kernel
$\bbz=\angles{\bft_4}$ needs to be ``inflated'' to $\tfrac1q
\bbz=\angles{\bft^{1/q}_4}$. It turns out that, after appropriate
inflation, an extension $\varPi$ always exists.

The abstract kernel of $\Phi \ra \out(\gamsig)$ is given by, for $A \in \Phi$, 
\[
\mu(\alpha): \gamsig \ra \gamsig, \text{ where } \alpha =
(\bft_1^{a_1}\bft_2^{a_2}\bft_3^{a_3},A) \in Q.
\]
Here $\mu(\alpha)$ denotes conjugation in $\Sol\rx D_4$.  Suppose in {\rm
Proposition \ref{solofShape}}, we have fixed the $c_i$, as well as set
$q=1$, thus fixing the lattice
\[
\wt \Gamma_{(\cals;1,n_1,n_2)} = \angles{\bft_1, \bft_2, \bft_3, \bft_4}
\hookrightarrow \Solof.
\]
For any generator $A \in \Phi$, let
\[
\alpha = (\bft_1^{a_1} \bft_2^{a_2} \bft_3^{a_3} \bft_4^{{a_4}}, A) = (a,A).
\]
the 
effect that conjugation by $\alpha$ has on $\wt \Gamma_{(\cals;1,n_1,n_2)}$. Note that conjugation by $\alpha$ is independent of $a_4$.
We have the relations:
\begin{align*}
&\alpha \bft_1 \alpha\inv=\bft_1^{m_1}\bft_2^{m_2}  \bft_4^{{v_1}},
\text{ where } {\xy{m_1}{m_2}=\varphi(\alphabar)\xy{1}{0}},\\
&\alpha \bft_2 \alpha\inv=\bft_1^{n_1}\bft_2^{n_2} \bft_4^{{v_2}},
\text{ where } {\xy{n_1}{n_2}=\varphi(\alphabar)\xy{0}{1}},\\
&\alpha\bft_3\alpha\inv=\bft_1^{w_1}\bft_2^{w_2}\bft_3^{\abar} \bft_4^{{v_3}},
\text{ where } {\xy{w_1}{w_2}=
\left(I-\cals^{\abar}\right)\xy{a_1}{a_2}},\\
&\alpha\bft_4 \alpha\inv={\bft_4^{\ahat}},
\end{align*}

\noindent
We will need the following lemma on the $v_i$.
\begin{lemma}
\label{relations}
The numbers $v_1$ and $v_2$ are rational. Furthermore, we can adjust $c_3$ so that $v_3$ is rational.
\end{lemma}
\begin{proof}	
Note that the image of $\wt \Gamma_{(\cals;1,n_1,n_2)}$ under conjugation by $\alpha$,
\[
\mu(\alpha)(\wt \Gamma_{(\cals;1,n_1,n_2)} ) = \alpha \wt \Gamma_{(\cals;1,n_1,n_2)} \alpha \inv ,
\]
is a lattice
of $\Solof$ lifting the standard lattice $\gamsig$ of $\Sol$.
	
All such lifts are given in Lemma \ref{determine-TT}. In equation (\ref{c1c2-equation}), 
we see that for any two solutions $c_1, c_2$ and $c'_1, c'_2$, both $c'_1 - c_1$ and $c'_2 - c_2$ must be rational. 
Thus $v_1$ and $v_2$ are rational numbers.
	
From Proposition \ref{autosol14}, $A \in \Phi \subseteq D_4$ can be viewed as an element of $\gltz$. 
The induced action of $A$ on $\calz(\Solof)$ is multiplication by $\ahat = \det(A)$, 
and the induced action of $A$ on $\Solof / \Nil \cong \bbr$ is multiplication by $\abar$. 
We need to understand the action of $A$ on the generator $\bft_3$ of $\wt \Gamma_{(\cals;1,n_1,n_2)}$. 
Let $\hat \bft_3$ denote $\bft_3$ with the $(1,3)$-slot set to be zero, so that $\bft_3 = \hat \bft_3 \bft_4^{c_3}$:
\begin{align*}
A(\bft_3) &= A(\hat \bft_3 \bft_4^{c_3}) = A(\hat \bft_3) A(\bft^{c_3}_4)
=\hat \bft_3^{\abar} \bft^{\ahat c_3}_4
=(\hat \bft_3^{\abar} \bft^{\abar c_3}_4) (\bft^{-\abar c_3}_4
\bft^{\ahat c_3}_4) \\
&= \bft_3^{\abar}  \bft^{(\ahat -\abar) c_3}_4 .
\end{align*}

In order to show
\begin{align}
\label{teh}
\alpha\bft_3\alpha\inv=\bft_1^{w_1}\bft_2^{w_2}\bft_3^{\abar} \bft_4^{{v_3}},
\text{ where } {\xy{w_1}{w_2}=
\left(I-\cals^{\abar}\right)\xy{a_1}{a_2}},
\end{align}
we need only consider two cases, either $a_3=\tfrac12$ or $a_3=0$. 

First, consider the case when $a_3=\frac{1}{2}$. Then $A$ must be diagonal, so that $\abar=+1$. 
By Corollary \ref{a-conj-zero3}, we can take $a_1=a_2=0$ so that $\alpha = (\bft_3^{\frac{1}{2}} \bft_4^{a_4}, A)$, so
\begin{align*}
\alpha\bft_3\alpha\inv &= \bft_3^{\frac{1}{2}} A(\bft_3 ) \bft_3^{-\frac{1}{2}} 
=\bft_3^{\frac{1}{2}} \bft_3^{\abar}  \bft^{(\ahat -\abar) c_3}_4 \bft_3^{-\frac{1}{2}} = \bft_3 \bft_4^{(\ahat-1)c_3}.
\end{align*}
Since $\ahat=\pm 1$, there is a choice of $c_3$ which makes $(\ahat-1)c_3 \in \bbq$.
	
\noindent
Now consider the case $a_3=0$, so that
$
\alpha = (\bft_1^{a_1} \bft_2^{a_2}  \bft_4^{a_4}, A)
$.
We compute:
	
\begin{align*}
\alpha\bft_3\alpha\inv &=\bft_1^{a_1} \bft_2^{a_2} A(\bft_3 ) \bft_2^{-a_2} \bft_1^{-a_1} 
= \bft_1^{a_1} \bft_2^{a_2} \left( \bft_3^{\abar}  \bft^{(\ahat -\abar) c_3}_4 \right) \bft_2^{-a_2} \bft_1^{-a_1}\\
&=\left(\bft_1^{a_1} \bft_2^{a_2} \bft_3^{\abar} \bft_2^{-a_2} \bft_1^{-a_1} \bft_3^{-\abar} \right) \bft_3^{\abar} \bft_4^{(\ahat-\abar) c_3}.
\end{align*}
Now by Lemma \ref{rational}, and using that $a_1$, $a_2$ are rational, we have 
\begin{align*}
\left( \bft_1^{a_1} \bft_2^{a_2} \bft_3^{\abar} \bft_2^{-a_2} \bft_1^{-a_1} \bft_3^{-\abar} \right) \bft_3^{\abar} \bft_4^{(\ahat-\abar) c_3} &= 
\left( \bft_1^{b_1}  \bft_2^{b_2} \bft_4^{u}  \right) \bft_3^{\abar}  \bft_4^{(\ahat-\abar) c_3} \\
&= \bft_1^{b_1}  \bft_2^{b_2} \bft_3^{\abar}  \bft_4^{u+(\ahat-\abar) c_3},
\end{align*}
for a rational number $u$. Equating this with equation (\ref{teh}), we obtain
\begin{align*}
\bft_1^{w_1}\bft_2^{w_2}\bft_3^{\abar} \bft_4^{{v_3}} =  \bft_1^{b_1}  \bft_2^{b_2} \bft_3^{\abar}  \bft_4^{u+(\ahat-\abar) c_3}.
\end{align*}
Now $w_1=b_1$ and $w_2=b_2$ is forced. Therefore,
$v_3 = u+(\ahat-\abar) c_3.$
Because $\ahat=\pm 1$, $\abar= \pm 1$, and $u$ is rational,
$c_3$ can always be chosen so that $v_3$ is rational.\qedhere
\end{proof}

\begin{proposition}
Let $Q \hra \Sol \rx D_4$ be a crystallographic group of $\Sol$ with lattice $\gamsig$. 
Then there exists a lattice $\lgamsig = \angles{\bft_1, \bft_2, \bft_3, \bft_4^{\frac{1}{q}}}$ of $\Solof$, 
projecting to $\gamsig$, for which
the abstract kernel $\Phi \ra \out(\gamsig)$
induces $\Phi \ra \out(\lgamsig)$.
\end{proposition}
\begin{proof}

For any integer $q>0$, we add a finer generator of the central direction to
the group  $\wt \Gamma_{(\cals;1,n_1,n_2)}$ to obtain $\angles{\wt
\Gamma_{(\cals;1,n_1,n_2)}, \bft_4^{\frac {1}{q}}} = \wt
\Gamma_{(\cals;q,q n_1,q n_2)}.$
	
Now, for each generator $A \in \Phi$, the $v_i$ in Proposition \ref{relations} are rational. 
Therefore, for $q$ large enough, $\wt \Gamma_{(\cals;q,q n_1,q n_2)}$ is invariant under conjugation by 
$(\bft_1^{a_1}\bft_2^{a_2}\bft_3^{a_3} \bft_4^{a_4},A)$, for each $A \in \Phi$. 
As this conjugation is independent of lift of $(\bft_1^{a_1}\bft_2^{a_2}\bft_3^{a_3},A) \in \Sol \rx D_4$ to 
$(\bft_1^{a_1}\bft_2^{a_2}\bft_3^{a_3} \bft_4^{a_4},A) \in \Solof \rx D_4$, with $m_1 = q n_1$ and $m_2 = q n_2$,
we obtain an abstract kernel $\Phi \ra \out(\lgamsig)$.
\qedhere
\end{proof}

\begin{proposition}
\label{inflation}
Let $Q\hookrightarrow \Sol\rx D_4$ be a crystallographic group of $\Sol$
containing lattice $\gamsig$. Assume that the abstract kernel $\Phi \ra \out(\gamsig)$
induces $\Phi \ra \out(\lgamsig)$. Then for some $p>0$, there exists $\varPi$ which fits the following commuting diagram
\[
\CD
@. 1 @. 1 @.  @.\\
@. @VVV @VVV @. @.\\
@. \frac1{pq}\bbz @=  \frac1{pq}\bbz @. @.\\
@. @VVV @VVV @. @.\\
1 @>>> \wt \Gamma_{(\cals;pq,p m_1,p m_2)} @>>> \varPi @>>> \Phi @>>> 1\\
@. @VVV @VVV @| @.\\
1 @>>> \gamsig @>>> Q @>>> \Phi @>>> 1\\
@. @VVV @VVV @. @.\\
@. 1 @. 1 @.  @.\\
\endCD
\]
\end{proposition}

\begin{proof}
Since
the center of $\wt \Gamma_{(\cals;q, m_1, m_2)}$ is $\frac1{q} \bbz$ and $\Phi$ is finite, $H^3(\Phi;\tfrac1{q} \bbz)$
is finite. This means the obstruction class to the existence of the
extension vanishes if we use $\frac1{pq} \bbz$ for the coefficients, for some
$p>0$. That is, it vanishes inside $H^3(\Phi;\tfrac1{pq} \bbz)$. Thus, with such
$pq$, the center of $\wt \Gamma_{(\cals;pq,p m_1,p m_2)}$ is $\frac1{pq} \bbz$, and an extension $\varPi$
exists.
\qedhere
\end{proof}

So we can assume that after appropriate inflation, there exists
an extension $\varPi$ with lattice $\lgamsig$, for some $q>0$. The Seifert
Construction will show that such an abstract extension actually embeds in
$\Solof \rx D_4$ as a crystallographic group. 
By taking $pq$ as a new $q$, we have:

\begin{theorem}
\label{seifert-const}
Let $\tgamsig=\lgamsig$ be a lattice of $\Solof$,
and
\[
1\lra \tgamsig \lra \varPi \lra \Phi \lra 1
\]
be an extension of $\tgamsig$ by a finite group $\Phi$ from Proposition \ref{inflation}
Then there exists an injective homomorphism
\[
\theta: \varPi\ra \Solof\rx D_4 \subset \Solof\rx\aut(\Solof)
\]
carrying $\tgamsig$ onto a standard lattice.
Such $\theta$ is unique up to conjugation by an element of
$\Solof\rx\aut(\Solof)$.
\end{theorem}

\begin{proof}
This is a consequence of the Seifert construction, since
$\Solof$ is completely solvable. We can apply
\cite[Theorem 7.3.2]{seif-book} with $G=\Solof$ and $W=\text{\{point\}}$. Since
$\Phi$ is finite,
the homomorphism $\varPi\ra\out(\gammt)\ra\out(\Solof)$ has
finite image in $\out(\Solof)$, and it lifts back to a finite subgroup $C$ of
$\aut(\Solof)$. But this $C$ can be conjugated into
$D_4\subset \aut(\Solof)$, a maximal compact subgroup.
Consequently, we have a commuting diagram
\[
\CD
1 @>>> \tgamsig @>>> \varPi @>>> \Phi @>>> 1\\
@. @VVV @VVV @VVV @.\\
1 @>>> \Solof   @>>> \Solof\rx D_4 @>>> D_4 @>>> 1
\endCD
\]
The homomorphism $\varPi \ra \Solof \rx D_4$ is injective since the abstract kernel $\Phi \ra \out(\tgamsig)$ 
from Proposition \ref{inflation} is injective.
The essence of the argument is showing that the cohomology set
$H^2(\Phi;\Solof)$ is trivial for any finite group $\Phi$.
The uniqueness is a result of \cite[Corollary 7.7.4]{seif-book}. It also
comes from $H^1(\Phi;\Solof)=0$.
\qedhere
\end{proof}
After inflation, the Seifert Construction produces a crystallographic group of $\Solof$.
Often we can assume that $c_3 = 0$, that is, $\lgamsig$ is
a standard lattice of $\Solof$.
Recall that $\aut(\Solof)=\bbr\rx\aut(\Sol)$ (Proposition \ref{autosol14}),
where $\hat k \in \bbr$ acts by
\[
\left[
\begin{matrix} 1 &e^u x &z\\
0 &e^u &y\\
0 &0 &1 \end{matrix}
\right]
\longmapsto
\left[
\begin{matrix} 1 &e^u x &z + k u \\
0 &e^u &y\\
0 &0 &1 \end{matrix}
\right].
\]
We have the following:

\begin{theorem}
\label{abstract-kernel-theorem}
For all holonomy groups,
except $\bbz_4$, a crystallographic group $\varPi$ of $\Solof$ embeds into
$\Solof\rx D_4$ in such a way that $\varPi\cap\Solof$ is a standard lattice {\rm ($c_3=0$)}.
	
\end{theorem}
\begin{proof}
Let $e$ denote the identity element of $\Solof$.
For the statement concerning $c_3$,
conjugation by $(e,\hat k)$ with $k=-\frac {c_3} {{\ln \lambda}}$ sets $c_3=0$ in $\bft_3$.
However, this conjugation moves  $D_4$ to $\hat k D_4 \hat k \inv$.
	
Suppose every $A\in \Phi$ satisfies $\abar\hat A=+1$.
Since such $A$ commute with $\hat k$,
conjugation by $(e,\hat k)$ leaves the holonomy group $\Phi$ inside $D_4$ while
setting $c_3=0$ in $\bft_3$.
This applies to, from the list of Theorem \ref{abstract-kernel},
all the groups lifting $\Sol$-crystallographic groups of type \gr{2a}, \gr{2b}, \gr{3}, \gr{3i}, \gr{6a}, \gr{6ai}, \gr{6b}, and
\gr{6bi}.
	
Suppose $\Phi$ contains $A=\matPpm$.  Then Corollary \ref{a-conj-zero3} and
Lemma \ref{conj-in-t4} below
show that a generator $\alpha$ of $\varPi$ projecting to $A \in \Phi$ can be conjugated to $\alpha=(\bft_3^{\frac12},A)$
(so that $a_1=a_2=a_4=0$).
Then, we shall show that
$\bft_3=\hat\bft_3\bft_4^{c_3}$ can be replaced by $\hat\bft_3$
(where $\hat\bft_3$ is $\bft_3$ with $c_3=0$).
\begin{align*}
\alpha^2
&=(\bft_3^{\frac12},A)^2
=((\hat\bft_3\bft_4^{c_3})^{\frac12},A)^2
=(\hat\bft_3\bft_4^{c_3})^{\frac12} A((\hat\bft_3\bft_4^{c_3})^{\frac12})\\
&=\hat\bft_3^{\frac12} \bft_4^{\frac{c_3}2} \cdot \hat\bft_3^{\frac12} \bft_4^{-\frac{c_3}2}
=\hat\bft_3.
\end{align*}
Thus $\hat\bft_3=\alpha^2\in\varPi$, and we can take $\hat\bft_3$ instead of
$\bft_3$ as a generator for the same group (which is
apparently redundant since $\alpha$ is in the group already).
This shows that 
$\bft_4^{c_3}=\alpha^{-2} \bft_3 \in\varPi$ must be a multiple of $\frac1q$, and
we can take $c_3=0$.
From the list in Theorem \ref{abstract-kernel},
the groups \gr{1}, \gr{5}, \gr{7} and \gr{7i}
contain such an $A$ in the holonomy.
	
The only case that is not covered by these two cases is when $\Phi=\bbz_4$
(type \gr{4} in the list), which is discussed below in our main classification (Theorem 
\ref{ClassificationSolof-geometry}).\qedhere
\end{proof}

\begin{lemma}
\label{conj-in-t4}
If $\det(A)=-1$, by conjugation, $a_4$ can be made 0.
\end{lemma}

\begin{proof}
Suppose $\det(A)=-1$. Conjugation by
$\bft_4^{-\frac{a_4}{2}}$ fixes the lattice $\lgamsig$,
and moves $(\bft_1^{a_1}\bft_2^{a_2}\bft_3^{a_3}\bft_4^{{a_4}},A)$ to
$(\bft_1^{a_1}\bft_2^{a_2}\bft_3^{a_3},A)$.
\end{proof}

\begin{proposition}[Fixing $a_4, b_4$]
Consider the commuting diagram in {\rm Proposition \ref{inflation}}.
Given $Q$ and integers $q, m_1, m_2$, we had $\lgamsig$. The only thing
that remains for the construction of $\varPi$ is fixing $a_4,b_4$. As is
known, all the extensions $\varPi$ in the short exact sequence
\[
1\ra\lgamsig\ra\varPi\ra \Phi\ra 1
\]
are classified by $H^2(\Phi;\calz(\lgamsig))=H^2(\Phi;\bbz)$.
When $\Phi=\angles{A}$,
\[
H^2(\bbz_p;\bbz)=
\begin{cases}
0,       &\text{if $\ahat=-1$};\\
\bbz_p,  &\text{if $\ahat=1$},
\end{cases}
\]
see \textup{\cite[Theorem 7.1, p.122]{maclane}}
.
	
In actual calculation, this becomes an equation
\[
\alpha^p=\bft_1^{n_1} \bft_2^{n_2} \bft_3^{n_3} \bft_4^{k_4}
\]
for integers $n_i$ and  $k_4=\tfrac{i}{q}$, $i=0,1,\cdots,p-1$.
\qed
\end{proposition}

\begin{remark}
When $\Phi=\angles{A,B}$ is not cyclic, $\ahat=\bhat=+1$ never happens, so we can set
one of $a_4$, $b_4$ to zero. Thus, $H^2(\Phi; \calz(\tgamsig))$ is cyclic for all $\Phi$.
\end{remark}

\step[{\bf Detecting Torsion in $\Solof$-Crystallographic Groups}]

Given a lattice $\wt \Gamma_\cals$ of $\Solof$ (which projects to a lattice $\gamsig$ of $\Sol$), the short exact sequence
\[
1 \ra \calz(\wt \Gamma_\cals) \ra \wt \Gamma_\cals \ra \gamsig \ra 1
\]
induces an $S^1$-bundle over the solvmanifold $\gamsig \backslash \Sol$,
\[
S^1 \ra \wt \Gamma_\cals \backslash (\Solof) \ra \gamsig \backslash \Sol.
\]
The following two lemmas will be useful for determining when a $\Solof$-crystallographic group is torsion free.

\begin{lemma}
\label{torSolof}
Let $\wt \Gamma_\cals$ be a lattice of $\Solof$, projecting
to a standard lattice $\gamsig$ of $\Sol$, and suppose that
for $\alpha \in  \Solof \rx D_4$, the group
$\varPi = \angles{\wt \Gamma_\cals, \alpha}$ is crystallographic.
Let $\bar \alpha$ denote the projection of $\alpha$ to $\Sol \rx D_4$.
When the automorphism part of $\alpha$ 
acts as a reflection on the center of $\Solof$, $\varPi$ is torsion free if and only if 
$\angles{\gamsig,\bar \alpha} \subset \Sol \rx D_4$ is torsion free.
\end{lemma}
\begin{proof}
Evidently, if $\angles{\gamsig,\bar \alpha}$ is torsion free, then $\varPi$
must be torsion free. For the converse, suppose that $\angles{\gamsig,\bar \alpha}$
has torsion. In this case, the action of $\bar \alpha$ on the solvmanifold
$\gamsig \bs \Sol$ must fix a point. Observe that the action of $\alpha$
on the solvmanifold $\wt \Gamma_\cals \bs \Solof$ is $S^1$ fiber preserving.
Therefore, a circle fiber is left invariant under the action of $\alpha$. Since
$\alpha$ acts as reflection on the fiber, $\alpha$ must fix a point. Since the action of $\alpha$ 
fixes a point on $\wt \Gamma_\cals \bs \Solof$, the action of $\varPi$ fixes a point on $\Solof$. Thus,
$\varPi$ has torsion.
\qedhere
\end{proof}

\begin{lemma}
\label{EasyTor}
Let $\varPi$ be a crystallographic group of $\Solof$ with lattice $\wt \Gamma_\cals$.
If $\alpha =(\bft_1^{a_1}\bft_2^{a_2}\bft_3^{a_3}\bft_4^{a_4}, A) \in \varPi$ satisfies
$a_3 = \frac{1}{2}$ and $\abar=1$, then $\gamma \alpha$ is infinite order for any $\gamma \in
\wt \Gamma_\cals$.
\end{lemma}
\begin{proof}
Note that $A$ is necessarily of order 2. Let $\pr: \Solof \ra \bbr$ denote the
quotient homomorphism of $\Solof$ by its nil-radical $\Nil$. Write $\gamma \in \wt \Gamma_\cals$
as $\bft_1^{n_1}\bft_2^{n_2}\bft_3^{n_3}\bft_4^{n_4}$. Application of $\pr$ to $(\gamma \alpha)^2$
yields
\[
\pr(\gamma \alpha)^2 = 2 n_3+1,
\]
from which we infer $\gamma \alpha$ is of infinite order.
\qedhere
\end{proof}

We are now ready to give our main classification of $\Solof$-crystallographic groups.
Following Proposition \ref{solofShape}, a crystallographic group 
\[
\varPi \subset \Solof \rx D_4
\]
 of $\Solof$ is generated
by a lattice $\lgamsig = \angles{\bft_1,\bft_2,\bft_3, \bft_4^{\frac1q}}$ of $\Solof$, 
together with at most two generators of the form
\begin{align*}
&(\bft_1^{a_1} \bft_2^{a_2} \bft_3^{a_3} \bft_4^{a_4} , A),
(\bft_1^{b_1} \bft_2^{b_2} \bft_3^{b_3} \bft_4^{b_4} , B),
\end{align*}
where $A,B$ generate the holonomy group $\varPhi \subset D_4$. The $\Solof$-crystallographic group 
$\varPi$ projects to a $\Sol$-crystallographic group $Q$. 
We view $Q$ as an extension
\[1 \ra \bbz^2 \ra Q \ra \bbzphi \ra 1,\]
and Theorem \ref{abstract-kernel-theorem} classifies all possible $\bbzphi$ and abstract kernels $\varphi: \bbzphi \ra \gltz$.
We organize the $\Solof$-crystallographic groups according to which $\bbzphi$ and 
$\varphi: \bbzphi \ra \gltz$ in Theorem \ref{abstract-kernel-theorem} they project to. 
Theorem \ref{ClassificationSolof-geometry} also classifies $\Sol$-crystallographic groups, by projecting from $\Solof \rx D_4$ to $\Sol \rx D_4$.

\begin{theorem}[Classification of $\Solof$-Crystallographic Groups]
\label{ClassificationSolof-geometry}
The following is a complete list of
crystallographic groups $\varPi$ of $\Solof$, generated by a lattice $\lgamsig$ of $\Solof$, 
together with at most two generators of the form \begin{align*}
&(\bft_1^{a_1} \bft_2^{a_2} \bft_3^{a_3} \bft_4^{a_4} , A),
(\bft_1^{b_1} \bft_2^{b_2} \bft_3^{b_3} \bft_4^{b_4} , B).
\end{align*}

They are organized according to which 
$\bbzphi$ and $\varphi: \bbzphi \ra \gltz$ they project to (see Theorem \ref{abstract-kernel-theorem}). 
This determines the exponents $a_3, b_3$.

We find equations describing $H^1(\Phi;\coker(I-\cals))$, and thus classifying $\bfa=\xy{a_1}{a_2}$, $\bfb=\xy{b_1}{b_2}$. 
In general, $H^1(\Phi;\coker(I-\cals))$ depends on $\cals$.

By Proposition \ref{inflation} and Theorem \ref{seifert-const}, for sufficiently large $q$, 
an abstract kernel $\Phi \ra \out(\lgamsig)$ is induced, with vanishing obstruction to the existence of $\varPi$ in $H^3(\Phi; \calz(\lgamsig))$.
The exponents on $\bft_4$, $a_4$ and $b_4$, are classified by the group 
$H^2(\Phi; \calz(\lgamsig))$.

In all cases, except, $\Phi = \bbz_4$, we can take $c_3=0$ in the lattice $\lgamsig$ of $\varPi$ (Theorem \ref{abstract-kernel-theorem}).
In the $\bbz_4$ holonomy case, we have two different (up to isomorphism) choices for $c_3$.

Whenever the holonomy group contains an automorphism of $\Solof$ which is represented by an
off-diagonal matrix, the orbifold $\varPi \bs \Solof$ is non-orientable.
We give precise criterion for $\varPi$ to be torsion free. When $\varPi$ is torsion free, $\varPi \bs \Solof$ is an infra-solvmanifold of $\Solof$.

By projecting each $\Solof$-crystallographic group $\varPi$ to a crystallographic group $Q \subset\Sol \rx D_4$, 
we also obtain a classification of $\Sol$-crystallographic groups.
\end{theorem}

\begin{enumerate}
\item[]
\bigskip
\item[\gr{0}]
$\Phi=\text{trivial}$ \parn
\[ 
\varPi=\wt\Gamma_{(\cals;q,m_1,m_2)}.
\]\parn
$\bullet$ Torsion free
\\
\item[\gr{1}]
$\Phi=\bbz_2$: $A=
\left[\begin{matrix}
1 &0\\ 0& -1
\end{matrix}\right]$,\parn
$\bbzphi=\bbz=\angles{\bft_3,\bar\alpha=(\bft_3^{\frac12},A)}$.\parn
$\vara=-K$ with $\det(K)=-1$,  $\tr(K)=n>0$, and $\cals=n K+I$.\parn

\[
\varPi=\angles{\bft_1,\bft_2,\bft_3,\bft_4^{\frac{1}{q}},
\alpha =(\bft_3^{\frac12},A)}.
\]\parn
$\bullet$ $H^1(\Phi;\coker(I-\cals))$ is trivial so that $\bfa=\bzero$. \parn
$\bullet$ $H^2(\Phi; \calz(\tgamsig))$ is trivial.\parn
$\bullet$ Both $Q$ and $\varPi$ are torsion free.
\\
\item[\gr{2a}]
$\Phi=\bbz_2$: $A=
\left[\begin{matrix}
-1 &0\\ 0& -1
\end{matrix}\right]$,\parn
$\bbzphi=\bbz\x\bbz_2=\angles{\bft_3,\bar\alpha=(\bft_3^{0},A)}$.\parn
$\vara=A$, $\cals\in\sltz$ with $\tr(\cals)>2$.\parn
\[	
\varPi=\angles{\bft_1,\bft_2,\bft_3,\bft_4^{\frac{1}{q}},
\alpha =(\bft_1^{a_1} \bft_2^{a_2}\bft_4^{a_4},A)}.
\]
$\bullet$ $H^1(\Phi;\coker(I-\cals))$
 {$= \{\bfa \mid \bfa \in \coker(I-\cals)\}/
\{2\bfa \mid \bfa \in \coker(I-\cals)\} \subseteq \bbz_2 \x \bbz_2$} \parn
$\bullet$ $H^2(\Phi; \calz(\tgamsig)) = \bbz_2$.  There are two choices for $a_4$, the solutions of 
$\alpha^2 = \bft_4^{\frac {i} {q}}$ ($i=0,1$). \parn
$\bullet$
$Q$ has torsion, $\varPi$ is torsion free when $i=1$ and $q$ is even.
\\
\item[\gr{2b}]
$\Phi=\bbz_2$: $A=
\left[\begin{matrix}
-1 &0\\ 0& -1
\end{matrix}\right]$,\parn
$\bbzphi=\bbz=\angles{\bft_3,\bar\alpha=(\bft_3^{\frac12},A)}$.\parn
$\vara=-K$ with $\det(K)=+1$,  $\tr(K)=n>2$, and $\cals=n K-I$.\parn
\[
\varPi=\angles{\bft_1,\bft_2,\bft_3,\bft_4^{\frac{1}{q}},
\alpha =(\bft_3^{\frac12}\bft_4^{a_4},A)}.
\]
$\bullet$ $H^1(\Phi;\coker(I-\cals))$ is trivial so that $\bfa=\bzero$. \parn
$\bullet$ $H^2(\Phi; \calz(\tgamsig))=\bbz_2$, $a_4=0$ or $\frac{1}{2q}$.\parn
$\bullet$ Both $Q$ and $\varPi$ are torsion free.
\\
\item[\gr{3}]
$\Phi=\bbz_2$: $A=
\left[\begin{matrix}
0 &1\\ 1& 0
\end{matrix}\right]$,\parn
$\bbzphi=\bbz\rx\bbz_2=\angles{\bft_3,\bar\alpha=(\bft_3^{0},A)}$.\parn
$\vara=A$,
$\cals\in\sltz$ with $\tr(\cals)>2$ and $\sigma_{12}=-\sigma_{21}$. \parn
\[
\varPi=\angles{\bft_1,\bft_2,\bft_3,\bft_4^{\frac{1}{q}},
\alpha =(\bft_1^{a_1} \bft_2^{a_2},A)}.
\]\parn
$\bullet$ $H^1(\Phi;\coker(I-\cals)) = \{ \bfa \mid \bfa \in
\coker(I-\cals), a_2\equiv -a_1\}/\\ 
\left\{{\small{{\xy{v_1-v_2}{v_2-v_1}}}} \mid \bfv \in \coker(I- \cals)
\right\} \subseteq \bbz_2$.\parn
$\bullet$ $H^2(\Phi; \calz(\tgamsig))$ is trivial.\parn
$\bullet$
Both $Q$ and $\varPi$ have torsion.
\\
\item[\gr{3\ii}]
$\Phi=\bbz_2$: $A=
\left[\begin{matrix}
0 &1\\ 1& 0
\end{matrix}\right]$,\parn
$\bbzphi=\bbz\rx\bbz_2=\angles{\bft_3,\bar\alpha=(\bft_3^{0},A)}$.\parn
$\vara=\left[\begin{matrix}
1&0\\ 0&-1
\end{matrix}\right]$,
$\cals\in\sltz$
with $\tr(\cals)>2$ and  $\sigma_{11}=\sigma_{22}$.\parn
\[
\varPi=\angles{\bft_1,\bft_2,\bft_3,\bft_4^{\frac{1}{q}},
\alpha =(\bft_1^{a_1} \bft_2^{a_2},A)}.
\]\parn
$\bullet$ $H^1(\Phi;\coker(I-\cals)) = \{ \bfa \mid \bfa \in \coker(I-
\cals), 2 a_1\equiv 0\}/ \\ \left\{ \xy{0}{2v_2} \mid \bfv \in \coker(I-
\cals) \right\} \subseteq \bbz_2 \x \bbz_2$.\parn
$\bullet$ $H^2(\Phi; \calz(\tgamsig))$ is trivial.\parn
$\bullet$
Both $Q$ and $\varPi$ are torsion free if and only if $a_1\equiv\tfrac12$
and $a_2\not\equiv \frac{(\sigma_{11}+1)(2 n+1)}{2\sigma_{12}}$
for any $n\in\bbz$.
\\
\par\noindent
\item[\gr{4}]
$\Phi=\bbz_4$: $A=
\left[\begin{matrix}
0 &1\\ -1& 0
\end{matrix}\right]$,\parn
$\bbzphi=\bbz\rx\bbz_4=\angles{\bft_3,\bar\alpha=(\bft_3^{0},A)}$.\parn
$\vara=A$,
$\cals\in\sltz$ with $\tr(\cals)>2$ and symmetric.\parn
\[
\varPi=\angles{\bft_1,\bft_2,\bft_3,\bft_4^{\frac{1}{q}},
\alpha =(\bft_1^{a_1} \bft_2^{a_2}\bft_4^{a_4},A)}.
\]\parn
$\bullet$
There are two choices for $c_3$ in $\bft_3$. They are solutions of
$d=0$ or $d=\tfrac1q$ for $c_3$, where
$\alpha\bft_3 \alpha\inv
=\bft_1^{(1-\sigma_{22})a_1+\sigma_{12}a_2}\bft_2^{
\sigma_{21}a_1+(1-\sigma_{11})a_2}\bft_3^{-1}\bft_4^{d}$. Each corresponds to 
a distinct abstract kernel $\Phi \ra \out(\tgamsig)$.\parn
$\bullet$ $H^1(\Phi;\coker(I-\cals)) = \{ \bfa \mid \bfa \in \coker(I-\cals)\}/ 
\{ (I-A)\bfa \mid \bfa \in \coker(I-\cals) \} \subseteq \bbz_2$.\parn
$\bullet$
$H^2(\Phi; \calz(\tgamsig)) = \bbz_4$. There are 4 choices for $a_4$, the solutions of $\alpha^4=\bft_4^{\frac{i}{q}}$,
$\quad (i=0,1,2,3)$.\parn
$\bullet$
$Q$ has torsion, $\varPi$ is torsion free precisely when $i=1,3$ and $q$ is even.
\\
\item[\gr{5}]
$\Phi=\bbz_2\x\bbz_2$:
$A=
\left[\begin{matrix}
1 &0\\ 0& -1
\end{matrix}\right]$,
$B=
\left[\begin{matrix}
-1 &0\\ 0& -1
\end{matrix}\right]$,\parn
$\bbzphi=\bbz\x\bbz_2=\angles{\bft_3,\bar\alpha=(\bft_3^{\frac12},A),
\bar\beta=(\bft_3^0,B)}$.\parn
$\vara=-K$, $\varb=B$
\hfill \gr{1}$+$\gr{2a}\parn
$\cals=n K+I$, $K\in\gltz$, $\det(K)=-1$, and  $\tr(K)=n>0$.\parn
\[
\varPi=\angles{\bft_1,\bft_2,\bft_3,\bft_4^{\frac{1}{q}},
\alpha=(\bft_3^{\frac12},A), 
\beta=(\bft_1^{b_1} \bft_2^{b_2} \bft_4^{b_4},B)}.
\]\parn
$\bullet$ $H^1(\Phi;\coker(I-\cals)) = \{ \bfb \mid \bfb \in \coker(I+K)\}/ 
\{ 2\bfb \mid  \bfb \in \coker(I+K)\} \subseteq \bbz_2 \x \bbz_2$.\parn
$\bullet$ $H^2(\Phi; \calz(\tgamsig))=\bbz_2$.
There are two choices for $b_4$, the solutions of
$\beta^2=\bft_4^{\frac{i}{q}}, (i=0,1)$.\parn
$\bullet$
$Q$ has torsion, $\varPi$ is torsion free precisely when $i=1$ and $q$ is even.
\\
\item[\gr{6a}]
$\Phi=\bbz_2\x\bbz_2$:
$A=
\left[\begin{matrix}
0 &1\\ 1& 0
\end{matrix}\right]$,
$B=
\left[\begin{matrix}
-1 &0\\ 0& -1
\end{matrix}\right]$,\parn
$\bbzphi=(\bbz\x\bbz_2)\rx\bbz_2=\angles{\bft_3,\bar\alpha=(\bft_3^{0},A),
\bar\beta=(\bft_3^0,B)}$.\parn
$\vara=A$, $\varb=B$ \hfill \gr{3}$+$\gr{2a}\parn
$\cals\in\sltz$ with $\tr(\cals)>2$ and $\sigma_{12}=-\sigma_{21}$. \parn
\[	
\varPi=\angles{\bft_1,\bft_2,\bft_3,\bft_4^{\frac{1}{q}},
\alpha=(\bft_1^{a_1} \bft_2^{a_2},A),
\beta=(\bft_1^{b_1} \bft_2^{b_2}\bft_4^{b_4},B)}.
\]\parn
$\bullet$ $H^1(\Phi;\coker(I-\cals)) = 
\{ (\bfa,\bfb) \mid \bfa, \bfb \in \coker(I-\cals), a_2\equiv -a_1, b_1-b_2-2 a_1 \equiv 0\}/
\left\{ \left({\scriptstyle{\xy{v_1-v_2}{v_2-v_1}}}, 2\bfv\right)   \mid  \bfv \in \coker(I-\cals) \right\}$.\parn
$\bullet$ $H^2(\Phi; \calz(\tgamsig))=\bbz_2$.
There are two choices for $b_4$, the solutions of
$\beta^2=\bft_4^{\frac{i}{q}}, (i=0,1)$.\parn
$\bullet$
Both $Q$ and $\varPi$ have torsion.
\\
\item[\gr{6a\ii}]
$\Phi=\bbz_2\x\bbz_2$:
$A=
\left[\begin{matrix}
0 &1\\ 1& 0
\end{matrix}\right]$,
$B=\def\tr{\text{\rm{tr}}}
\left[\begin{matrix}
-1 &0\\ 0& -1
\end{matrix}\right]$,\parn
$\bbzphi=(\bbz\x\bbz_2)\rx\bbz_2=\angles{\bft_3,\bar\alpha=(\bft_3^{0},A),
\bar\beta=(\bft_3^0,B)}$.\parn
$\vara=
\left[\begin{matrix}
1 &0\\ 0& -1
\end{matrix}\right]$,\
$\varb=B$
\hfill \gr{3\ii}$+$\gr{2a}\parn
$\cals\in\sltz$
with $\tr(\cals)>2$ and $\sigma_{11}=\sigma_{22}$.\parn
\[	
\varPi=\angles{\bft_1,\bft_2,\bft_3,\bft_4^{\frac{1}{q}},
\alpha=(\bft_1^{a_1} \bft_2^{a_2},A),
\beta=(\bft_1^{b_1} \bft_2^{b_2}\bft_4^{b_4},B)}.
\]\parn
$\bullet$ $H^1(\Phi;\coker(I-\cals)) = \{ (\bfa,\bfb) \mid \bfa, \bfb \in \coker(I-\cals), 2 a_1\equiv 0, 2b_2-2a_2 \equiv 0\}/ 
\left\{ \left({\xy{0}{2v_2}}, 2\bfv\right)   \mid  \bfv \in \coker(I-\cals) \right\}$.\parn
$\bullet$ $H^2(\Phi; \calz(\tgamsig))=\bbz_2$.
There are two choices for $b_4$, the solutions of
$\beta^2=\bft_4^{\frac{i}{q}}, (i=0,1)$.\parn
$\bullet$
$Q$ has torsion, 	$\varPi$ is torsion free if and only if
$i=1$, $q$ is even, and
$a_1\equiv\tfrac12$, $a_2\equiv b_2 + \frac 12$,
$b_1\not\equiv \frac {\sigma_{12}(2 n+1)}{2(\sigma_{11}-1)} + \frac 1 2$,
$b_2\not\equiv \frac{(\sigma_{11}+1)(2 m+1)}{2\sigma_{12}} + \frac 1 2$ for any $m,n \in \bbz$.
\\
\item[\gr{6b}]
$\Phi=\bbz_2\x\bbz_2$:
$A=
\left[\begin{matrix}
0 &1\\ 1& 0
\end{matrix}\right]$,
$B=
\left[\begin{matrix}
-1 &0\\ 0& -1
\end{matrix}\right]$,\parn
$\bbzphi=\bbz\rx\bbz_2=\angles{\bft_3,\bar\alpha=(\bft_3^{0},A),
\bar\beta=(\bft_3^{\frac12},B)}$.\parn
$\vara=A$, $\varb=-K$ \hfill \gr{3}$+$\gr{2b}\parn
$\cals=n K-I$, where $K\in\sltz$ with $\tr(K)= n>2$; $k_{12}=-k_{21}$.\parn
\[
\varPi=\angles{\bft_1,\bft_2,\bft_3,\bft_4^{\frac{1}{q}},
\alpha=(\bft_1^{a_1} \bft_2^{a_2},A),
\beta=(\bft_3^{\frac12}\bft_4^{b_4},B)}.
\]\parn
$\bullet$ $H^1(\Phi;\coker(I-\cals)) = \{ \bfa \mid \bfa \in \coker(I+K), a_2\equiv -a_1\}/ \\ 
\left\{ {\scriptstyle{\xy{v_1-v_2}{v_2-v_1}}} \mid \bfv \in \coker(I+K) \right\} \subseteq \bbz_2$.\parn
$\bullet$ $H^2(\Phi; \calz(\tgamsig))=\bbz_2$.
There are two choices for $b_4$, the solutions of
$\beta^2=\bft_3 \bft_4^{\frac{i}{q}}, (i=0,1)$.\parn
$\bullet$
Both $Q$ and $\varPi$ have torsion.
\\
\item[\gr{6b\ii}]
$\Phi=\bbz_2\x\bbz_2$:
$A=
\left[\begin{matrix}
0 &1\\ 1& 0
\end{matrix}\right]$,
$B=
\left[\begin{matrix}
-1 &0\\ 0& -1
\end{matrix}\right]$,\parn
$\bbzphi=\bbz\rx\bbz_2=\angles{\bft_3,\bar\alpha=(\bft_3^{0},A),
\bar\beta=(\bft_3^{\frac12},B)}$.\parn
$\vara=\left[\begin{matrix}
1 &0\\ 0& -1
\end{matrix}\right]$, $\varb=-K$ \hfill \gr{3\ii}$+$\gr{2b}\parn
$\cals=n K-I$, where $K\in\sltz$ with $\tr(K)= n>2$;
$k_{11}=k_{22}$.\parn
\[
\varPi=\angles{\bft_1,\bft_2,\bft_3,\bft_4^{\frac{1}{q}},
\alpha=(\bft_1^{a_1} \bft_2^{a_2},A),
\beta=(\bft_3^{\frac12}\bft_4^{b_4},B)}.
\]\parn
$\bullet$ $H^1(\Phi;\coker(I-\cals)) = \{ \bfa \mid \bfa \in \coker(I+K), 2 a_1\equiv 0\}/ \\
\left\{ \xy{0}{2v_2} \mid \bfv \in \coker(I+K) \right\} \subseteq \bbz_2 \x \bbz_2$.\parn
$\bullet$
$H^2(\Phi; \calz(\tgamsig))=\bbz_2$.
There are two choices for $b_4$, the solutions of
$\beta^2=\bft_3\bft_4^{\frac{i}{q}}, (i=0,1)$.\parn
$\bullet$
Both $Q$ and $\varPi$ are torsion free if and only if
$a_1=\tfrac12$
and $a_2\not\equiv \frac{(k_{11}-1)(2 n+1)}{2k_{12}}$ for any $n \in \bbz$.
\\
\item[\gr{7}]
$\Phi=\bbz_4\rx\bbz_2$:
$A=\matNpp, B=\matPpm,$\parn
$\bbzphi=(\bbz\x\bbz_2)\rx\bbz_2=\angles{\bft_3,\bar\alpha=(\bft_3^{0},A),
\bar\beta=(\bft_3^{\frac12},B)}$.\parn
$\vara=A$, $\varb=-K$ \hfill (includes \gr{6a})\qquad  \gr{3}$+$\gr{1}\parn
$\cals=n K+I$, $K\in\gltz$, $\det(K)=-1$, $\tr(K)>0$;
$k_{12}=-k_{21}$.\parn
\[	
\varPi=\angles{\bft_1,\bft_2,\bft_3,\bft_4^{\frac{1}{q}},
\alpha=(\bft_1^{a_1} \bft_2^{a_2},A),
\beta=(\bft_3^{\frac12}\bft_4^{b_4},B)}.
\]\parn
$\bullet$ $H^1(\Phi;\coker(I-\cals)) = \{ \bfa \mid \bfa \in \coker(I-\cals), a_2\equiv -a_1\}/ \\
\left\{ {\scriptstyle{\xy{v_1-v_2}{v_2-v_1}}} \mid \bfv \in \coker(I+K) \right\}$.\parn
$\bullet$
$H^2(\Phi; \calz(\tgamsig))=\bbz_4$.
There are $4$ choices for $b_4$, the solutions of
$(\beta\alpha)^4=\bft_4^{\frac{j}{q}}, (j=0,1,2,3)$.\parn
$\bullet$
Both $Q$ and $\varPi$ have torsion.
\\	
\item[\gr{7\ii}]
$\Phi=\bbz_4\rx\bbz_2$:
$A=\matNpp, B=\matPpm$,\parn
$\bbzphi=(\bbz\x\bbz_2)\rx\bbz_2=\angles{\bft_3,\bar\alpha=(\bft_3^{0},A),
\bar\beta=(\bft_3^{\frac12},B)}$.\parn
$\vara=\matPpm$, $\varb =-K$
\hfill (includes \gr{6a\ii})\qquad \gr{3\ii}$+$\gr{1}\parn
$\cals=n K+I$, $K\in\gltz$, $\det(K)=-1$,  $\tr(K)=n>0$, $k_{11}=k_{22}$.\parn
\[
\varPi=\angles{\bft_1,\bft_2,\bft_3,\bft_4^{\frac{1}{q}},
\alpha=(\bft_1^{a_1} \bft_2^{a_2},A),
\beta=(\bft_3^{\frac12}\bft_4^{b_4},B)}.
\]\parn
$\bullet$ $H^1(\Phi;\coker(I-\cals)) = \{ \bfa \mid \bfa \in \coker(I-\cals), 2 a_1\equiv 0\}/ \\
\left\{ \xy{0}{2v_2} \mid \bfv \in \coker(I+K) \right\}$.\parn
$\bullet$
$H^2(\Phi; \calz(\tgamsig))=\bbz_4$.
There are $4$ choices for $b_4$, the solutions of
$(\beta\alpha)^4=\bft_4^{\frac{j}{q}}, (j=0,1,2,3)$.\parn
$\bullet$
$Q$ has torsion, $\varPi$ is torsion free if and only if $j=1,3$, $q$ is even, and
$a_1 = \frac 12$ and $a_2= -\frac {k_{21}+1}{2k_{11}} + \frac i {k_{11}}$ for $i=0, \ldots, k_{11}-1$.
\par\noindent
\end{enumerate}
\bigskip

\begin{proof}
	
Consider the descriptions of
\[
H^1(\Phi; \coker(I-\cals)) \cong Z^1(\Phi; \coker(I-\cals))/B^1(\Phi; \coker(I-\cals)) 
\] 
in Remark \ref{cocycleConditions}. 
In our computations below, we use that the condition  
\[
\bfa=\xy{a_1}{a_2} \in \coker(I-S) = (I-S)\inv \bbz^2 / \bbz^2
\]
is equivalent to
$(I-\cals) \bfa \equiv \bfzero \mbox{ mod } \bbz^2$.
 
In cases \gr{2a}, \gr{2b} and \gr{4}, $\Phi=\bbz_p$, $p=2$ or $4$.
Since $\det(A)=+1$, $\alpha^p$ has $\bft_4$ component ${\bft_4}^{p\cdot
a_4+\ell}$, where $\ell$ is independent of $a_4$. We then have $p$ choices for $a_4$ (modulo $\frac1q \bbz$). 
Namely, the solutions of 
\[
p\cdot a_4+\ell=\tfrac{1}{q}, \ldots, \tfrac{p-1}{q},
\]
each corresponding to a different class in $H^2(\Phi; \calz(\tgamsig))$. In fact, the number $\ell$ is always a rational number, 
and hence so is
$a_4$ (or $b_4$).
The remaining cases when $\Phi=\bbz_2\x\bbz_2$ or $D_4$ are similar. We set one of exponents on
$\bft_4$ by Lemma \ref{conj-in-t4}, and apply the above technique to find the remaining exponent on $\bft_4$.

\bigskip
\noindent
\gr{0}
See Theorem \ref{lattice-part-construction}.
\bigskip
	
\noindent
\gr{1} 
Corollary \ref{a-conj-zero3} shows $H^1(\Phi;\coker(I-\cals))$ is trivial,
and thus we can take $a_1=a_2=0$. Since $\ahat=\det(A)=-1$, Lemma \ref{conj-in-t4} implies
$a_4$ can be conjugated to zero.
So, $\varPi=\angles{\bft_1,\bft_2,\bft_3,\bft_4,
\alpha =(\bft_3^{\frac12},A)}$. By Lemma \ref{EasyTor}, both $\varPi$ and $Q$ are torsion free.

\bigskip
	
\noindent
\gr{2a}
In this case $\varphi(\alphabar)=-I$.
Now $\bfa$ must satisfy $(I-\cals)\bfa\equiv \bfzero$
taken modulo $(I-\varphi(\bar\alpha))\bfa=2\bfa$,
since the cocycle condition in Remark \ref{cocycleConditions}, $(I+\varphi(\bar\alpha))\bfa=\bfzero\in\bbz^2$, 
is satisfied independently of $\bfa$. Note that all elements of 
$H^1(\Phi; \coker(I-S))$ are of order 2, and is generated by at most 2 elements. Therefore,
$H^1(\Phi; \coker(I-S))$ is isomorphic to a subgroup of $\bbz_2 \x \bbz_2$.

There are two choices for $a_4$, the solutions of
$\alpha^2=\bft_4^{\frac{i}{q}}, (i=0,1)$. Indeed, $\alpha^2$ projects to the identity on $\Sol$.
Therefore, $\varPi$ is torsion free only when $i=1$ and $q$ is even (see classification of crystallographic groups
of $\Nil$, case 2, p.160, \cite{Dekimpe}), and $Q$ always has torsion.
\bigskip
	
\noindent
\gr{2b}
By Corollary \ref{a-conj-zero3}, we can take $a_1=a_2=0$ so that
$\alpha=(\bft_3^{\frac12}\bft_4^{a_4},A)$. Then
$\alpha^2=\bft_3\bft_4^{2 a_4}$.
Therefore, $a_4=0$ or $\frac{1}{2q}$.
By Lemma \ref{EasyTor}, both $\varPi$ and $Q$ are torsion free.
\bigskip
	
\noindent
\gr{3}
From Remark \ref{cocycleConditions},
$\bfa$ must satisfy $(I-\cals)\bfa\equiv\bfzero$, and
\[
(I+\vara)\bfa=
\left[\begin{matrix}
1&1\\ 1&1
\end{matrix}\right] \bfa\equiv\bfzero
\text{ modulo }
(I-\vara)\bfv=
\left[\begin{matrix}
1&-1\\ -1&1
\end{matrix}\right] \bfv
\text{, for }
(I-\cals)\bfv \equiv \bfzero.
\]
Computing, we obtain $a_2 \equiv -a_1$, modulo $\xy{v_1-v_2}{v_2-v_1}$. Applying the coboundary operator
to the cocycles yields:
\[
(I-\vara) \xy{a_1}{-a_1} = \xy{2a_1}{-2a_1},
\]
which implies that $\bfa$ has order at most 2 and so $H^1(\Phi;\coker(I-\cals))$ is either $\bbz_2$ or trivial, 
depending on $\coker(I-\cals)$.
By Lemma \ref{conj-in-t4}, we may assume $a_4=0$, equivalently,
$H^2(\Phi; \calz(\tgamsig))$ vanishes, so that $\alpha=(\bft_1^{a_1} \bft_2^{a_2} \bft_4^{a_4},A)$.

Direct computation shows that the projection of $\varPi$ to a $\Sol$-crystallographic group, $Q$, 
always has torsion. Note that $a_2 \equiv -a_1$, and
\begin{align*}
\alpha^2
&=(\bft_1^{a_1}\bft_2^{-a_1},A)^2
=(\bft_1^{a_1}\bft_2^{-a_1}\cdot
A(\bft_1^{a_1}\bft_2^{-a_1}),I) \\
&=(\bft_1^{a_1}\bft_2^{-a_1}\cdot \bft_2^{a_1}\bft_1^{-a_1},I)\\
&=(e,I).
\end{align*}

On $\Solof$, $\ahat=-1$, so $A$ acts as reflection on $\calz(\Solof)$. Lemma \ref{torSolof} 
applies to show that $\varPi$ always has torsion.
\bigskip
	
\noindent
\gr{3\ii}
From Remark \ref{cocycleConditions},
$\bfa$ must satisfy $(I-\cals)\bfa\equiv\bfzero$,
\[
(I+\vara)\bfa=
\left[\begin{matrix}
2&0\\ 0&0
\end{matrix}\right] \bfa\equiv\bfzero
\text{ modulo }
(I-\vara)\bfv=
\left[\begin{matrix}
0&0\\ 0&2
\end{matrix}\right] \bfv
\text{, for }
(I-\cals)\bfv\equiv\bfzero,
\]
that is, $2 a_1 \equiv \bfzero$ (so $a_1 \equiv 0$ or $\tfrac12$), modulo $\xy{0}{2v_2}$. 
This implies that $H^1(\Phi; \coker(I-\cals))$ is isomorphic to a subgroup of  $\bbz_2 \x \bbz_2$.
By Lemma \ref{conj-in-t4}, we may assume $a_4=0$, that is, $H^2(\Phi; \calz(\tgamsig))$ vanishes.
Therefore, $\alpha=(\bft_1^{a_1} \bft_2^{a_2},A)$.

Lemma \ref{torSolof} applies to show that $\varPi$ is torsion free precisely when the $\Sol$-crystallographic group $Q$ is torsion free,
which is equivalent to the action of $Q$ on $\Sol$ having no fixed points. By Lemma \ref{geometricLemma}, $Q \bs \Sol$ is $T^2 \x I$ with
$T^2 \x \{0 \}$ identified to itself by the affine involution of $T^2$ $\left( \xy{a_1} {a_2}, \matPpm \right)$ and $T^2 \x \{1 \}$ 
identified to itself by the affine involution
$\left( \xy{a_1} {a_2}, \left[ \begin{matrix} \sigma_{11} & -\sigma_{12}\\
\sigma_{21} & -\sigma_{11} \end{matrix} \right] \right)$. Both of these involutions act freely on the torus precisely when $a_1\equiv\tfrac12$
and $a_2\not\equiv \frac{(\sigma_{11}+1)(2 n+1)}{2\sigma_{12}}$
for any $n\in\bbz$.
\bigskip	

\noindent
\gr{4}
This is the only case where a non-standard lattice is present, that is
$c_3 \neq 0$.

By Remark \ref{cocycleConditions}, $\bfa$ must satisfy
$(I-\cals)\bfa \equiv \bfzero$, taken modulo
$\im(I-\varphi(\alphabar))$. Note that $\det(I-\varphi(\alphabar))=\det\left(\left[\begin{matrix}
1&-1\\ 1&1
\end{matrix}\right]\right)=2$, which implies
that $H^1(\Phi; \coker(I-S))$ is either $\bbz_2$ or the trivial group.

We compute that

\begin{align*}
\alpha\bft_3 \alpha\inv
&=\bft_1^{(1-\sigma_{22})a_1+\sigma_{12}a_2}\bft_2^{
\sigma_{21}a_1+(1-\sigma_{11})a_2}\bft_3^{-1}\bft_4^{u_4+2 c_3}.
\end{align*}
By Proposition \ref{relations}, $u_4$ must be rational. We have two choices for $c_3$ 
(modulo $\tfrac1q \bbz$, as $\bft_4^{\frac1q}$ is a generator of the lattice),
$ c_3 = -\tfrac{u_4}{2}, -\tfrac{u_4}{2} + \tfrac{1}{2q}, $
so that $u_4+2 c_3=0$ or $1$. Unless $c_3$ is a multiple of $\tfrac1q$, the corresponding lattice is non-standard.

For $a_4$, we have
\[
\alpha^4=\bft_4^{4a_4-(a_1-a_2)^2+v_4},
\]
Then there are 4 choices for $a_4$,
$a_4=\frac{(a_1-a_2)^2-v_4+i}{4q}, (i=0,1,2,3)$.
These are the solutions of
$
\alpha^4=\bft_4^{\frac{i}{q}}, \quad (i=0,1,2,3).
$

From this, $Q$ must always have torsion. For $i=0, 2$, $\varPi$ has torsion. To see this when $i=2$, note that
\[ 
 (\bft_4^{-\frac{1}{q}} \alpha^2)^2 = \bft_4^{-\frac{2}{q}}\bft_4^{\frac{2}{q}}=e.
\]
For $i=1,3$ and $q$ even (see classification of crystallographic groups
of $\Nil$, case 10, p.163, \cite{Dekimpe}), $\varPi$ is torsion free.

\bigskip
	
\noindent
\gr{5}
By Corollary \ref{a-conj-zero3}, we take $a_1=a_2=0$.
We need $\bfb$ to satisfy $(I-\cals)\bfb \equiv \bfzero$.
Then the cocycle conditions for
$\bbz_2 \x \bbz_2$ in Remark \ref{cocycleConditions}
show that we must have $(I - \varphi(\alphabar))\bfb$ = $(I+K)\bfb \equiv \bfzero$.
In fact, since $(I-\cals) = (I-K)(I+K)$,
this condition implies $(I-\cals)\bfb \equiv \bfzero$. Since
we have already fixed $a_1=a_2=0$, for the coboundary in Remark \ref{cocycleConditions}, we take $\bfb$ modulo
$(I - \varphi(\betabar))\bfv = 2 \bfv$ only when
$\bfv$ satisfies $(I - \varphi(\alphabar))\bfv$ = $(I+K)\bfv \equiv \bfzero$.

Since $\det(A)=-1$, we may assume $a_4=0$ by Lemma \ref{conj-in-t4}.
There are two choices for $b_4$, the solutions of
$\beta^2=\bft_4^{\frac{i}{q}}, (i=0,1)$, just like in case \gr{2a}. That is, $H^2(\Phi, \calz(\tgamsig)) = \bbz_2$. 
Indeed, $\beta$ has order 2 when projected to $\Sol \rx D_4$, and hence $Q$ always has torsion.

Note that $\gamma \alpha$ and $\gamma \alpha \beta$ are of infinite order for all $\gamma \in \wt \Gamma_\cals$ by Lemma \ref{EasyTor}.
Like case \gr{2a}, $\varPi$ is torsion free precisely when
$\beta^2=\bft_4^{\frac{1}{q}}$ and $q$ is even.
\bigskip
	
\noindent
\gr{6a} This is a combination of cases \gr{3}+\gr{2a}.

We have $(I-\cals)\bfa\equiv \bfzero$ and $(I-\cals)\bfb\equiv \bfzero$.
Also, $\bfa$ and $\bfb$ must satisfy the cocycle conditions
in Remark \ref{cocycleConditions}. Note that $(I+\varphi(\alphabar))\bfa \equiv \bfzero$
forces $a_2\equiv -a_1$, whereas 
\[
(I-\varphi(\alphabar)) \bfb-(I-\varphi(\betabar)) \bfa \equiv \bfzero
\] 
forces
$b_1-b_2-2 a_1 \equiv 0$, $-b_1 + b_2 - 2 a_2 \equiv 0$. Since $a_2\equiv -a_1$, the second equation is redundant. 
We take $\bfa$ and $\bfb$
modulo $(I-\varphi(\alphabar)) \bfv$ and $(I-\varphi(\betabar))\bfv$, respectively, where $(I-\cals)\bfv \equiv \bfzero$. 
By Lemma \ref{conj-in-t4}, we may assume $a_4=0$. There are two choices for $b_4$, the solutions of
$\beta^2=\bft_4^{\frac{i}{q}}, (i=0,1)$. That is, $H^2(\Phi, \calz(\tgamsig)) = \bbz_2$.

As $\varPi$ contains a subgroup of type \gr{3}, both $Q$ and $\varPi$ always have torsion.

\bigskip

\noindent
\gr{6a\ii} Similar to case \gr{6a}, this is a combination of  \gr{3\ii}+\gr{2a}.
The description of $H^1(\Phi, \coker(I-\cals))$ follows just like
in case \gr{6a}.

There are two choices for $b_4$, the solutions of
$\beta^2=\bft_4^{\frac{i}{q}}, (i=0,1)$. That is, $H^2(\Phi, \calz(\tgamsig)) = \bbz_2$. 
Since $\beta^2$ projects to the identity on $\Sol$, $Q$ always has torsion.

For $\varPi$ to be torsion free,
the subgroups
$\angles{\tgamsig,\alpha}$,
$\angles{\tgamsig,\beta}$, and
$\angles{\tgamsig,\alpha \beta}$, where
\[
\alpha \beta = \left(\bft_1^{a_1 + b_1} \bft_2^{a_2 - b_2} \bft_4^{b'_4}, \left[ \begin{matrix}0 & -1 \\ -1 & 0 \end{matrix} \right] \right),
\]
must all be torsion free.
\noindent
The group $\angles{\tgamsig,\beta}$ is torsion free precisely when $b_4$ satisfies
$\beta^2=\bft_4^{\frac{1}{q}}$ and $q$ is even.
	
By Lemma \ref{torSolof}, $\angles{\tgamsig,\alpha}$ and 	$\angles{\tgamsig,\alpha \beta}$ are torsion free precisely when
their projections to $\Sol$, $\angles{\gamsig, (\bft_1^{a_1} \bft_2^{a_2},A)}$ and
 $\angles{ \gamsig, (\bft_1^{a_1+b_1} \bft_2^{a_2-b_2},AB)}$ are torsion free.
Similar to case \gr{3\ii}, by computing when
the appropriate affine involutions on $T^2$ in Lemma \ref{geometricLemma}
have no fixed points, we obtain the conditions
$a_1=\frac 12$, $a_2 = b_2 + \frac 12$,
$b_1\not\equiv \frac {\sigma_{12}(2 n+1)}{2(\sigma_{11}-1)} + \frac 1 2$
$b_2\not\equiv \frac{(\sigma_{11}+1)(2 m+1)}{2\sigma_{12}} + \frac 1 2$ for any $m,n \in \bbz$.
\bigskip
	
\noindent
\item[\gr{6b}] This is a combination of  \gr{3}+\gr{2b}.

By Corollary \ref{a-conj-zero3}, we can take $b_1=b_2=0$.
The cocycle conditions in Remark \ref{cocycleConditions} force
$(I + \varphi(\alphabar)) \bfa \equiv \bfzero$ as well as
$(I - \varphi(\betabar)) \bfa = (I+K)\bfa \equiv \bfzero$, so that
$\bfa \in \coker(I+K)$.
Since $b_1, b_2 =0$ is fixed, we can take $\bfa$ modulo $(I -\varphi(\alphabar))\bfv$ only when
$(I-\varphi(\betabar))\bfv = (I+K)\bfv \equiv \bfzero$, that is, only for $\bfv \in \coker(I+K)$. 

Note that we can take $a_4=0$ by Lemma \ref{conj-in-t4}, and there are two choices for $b_4$, the solutions of
$\beta^2=\bft_3\bft_4^{\frac{i}{q}}, (i=0,1)$. Hence $H^2(\Phi, \calz(\tgamsig)) = \bbz_2$. 
Both $Q$ and $\varPi$ always have torsion, as they contain a subgroup of type \gr{3}.

\bigskip

\noindent
\item[\gr{6b\ii}] This is a combination of  \gr{3\ii}+\gr{2b}

By Corollary \ref{a-conj-zero3}, we can take $b_1=b_2=0$.
The computation of $H^1(\Phi;\coker(I-\cals))$
is identical to that of \gr{6b}. In this case, we use
$\varphi(\alphabar) = \matPpm$ rather than $\varphi = A$.
Note that we take $a_4=0$ by Lemma \ref{conj-in-t4}, and
there are two choices for $b_4$, the solutions of
$\beta^2=\bft_3\bft_4^{\frac{i}{q}}, (i=0,1)$. Thus $H^2(\Phi, \calz(\tgamsig)) = \bbz_2$. 

By Lemma \ref{torSolof}, $\varPi$ is torsion free precisely when the $\Sol$-crystallographic group $Q$ is torsion free, 
which is equivalent to $Q$ acting freely on $\Sol$. By Lemma \ref{geometricLemma}, $Q \bs \Sol$ is $T^2 \x I$ with
$T^2 \x \{0 \}$ identified to itself by the affine involution of $T^2$
$\left( \xy{a_1} {a_2}, \matPpm \right)$, and $T^2 \x \{1 \}$ identified to itself by the affine involution 
$\left( \xy{a_1} {a_2}, \left[ \begin{matrix} -k_{11} & k_{12}\\
-k_{21} & k_{11} \end{matrix} \right] \right)$. Both of these involutions act freely on the torus precisely when $a_1=\tfrac12$
	and $a_2\not\equiv \frac{(k_{11}-1)(2 n+1)}{2k_{12}}$ for any $n \in \bbz$.

\bigskip
	
\noindent
\item[\gr{7}] This is a combination \gr{3}+\gr{1}. which includes \gr{6a}.

By Corollary \ref{a-conj-zero3}, we can take $b_1=b_2=0$.
For $(I-\cals)\bfa \equiv \bfzero$, the only cocycle
condition that $\bfa$ must satisfy is $(I+\varphi(\alphabar))\bfa \equiv \bfzero$, which
forces $a_2\equiv -a_1$.
\noindent
However, we have fixed $b_1=b_2=0$. Therefore, when computing the coboundaries,
we can take $\bfa$ modulo $(I-\varphi(\alphabar))\bfv$ only for
$\bfv$ that satisfies
$(I-\varphi(\betabar))\bfv=(I+K)\bfv \equiv \bfzero$. Note that $(I+K)\bfv \equiv \bfzero$
actually implies  $(I-\cals)\bfv \equiv \bfzero$ since $(I-\cals) = (I-K)(I+K)$.

We may take $a_4=0$ by Lemma \ref{conj-in-t4}.
The computation
\[
(\beta\alpha)^4=\bft_4^{4b_4+\ell}
\]
shows that
there are $4$ choices for $b_4$, the solutions of
$(\beta\alpha)^4={\bft_4}^{\frac{j}{q}}, (j=0,1,2,3)$. Hence $H^2(D_4; \calz(\tgamsig)) = \bbz_4$.
Both $\varPi$ and $Q$ contain a subgroup of type \gr{3}, and so both always have torsion.
\bigskip

\noindent
\item[\gr{7\ii}] This is a combination of \gr{3\ii}+\gr{1}, which includes \gr{6a\ii}.

By Corollary \ref{a-conj-zero3}, we can take $b_1=b_2=0$. 
The description for $H^1(\Phi; \coker(I-\cals))$ follows as in case \gr{7},
using $\varphi(\alphabar) = \matPpm$ rather than $\varphi(\alphabar)=A$.

Like case \gr{7}, by Lemma \ref{conj-in-t4}, we take $a_4=0$, and there are $4$ choices for $b_4$, the solutions of
$(\beta\alpha)^4={\bft_4}^{\frac{j}{q}}, (j=0,1,2,3)$, so that $H^2(D_4; \calz(\tgamsig)) = \bbz_4$.

For $\varPi$ to be
torsion free, $\angles{\tgamsig, \beta\alpha}$ is necessarily
torsion free. This forces $b_4$ to satisfy $(\beta\alpha)^4=\bft_4^{\frac{j}{q}}
(j =1,3)$, and $q$ even. Note that $\angles{\tgamsig, \beta}$ and $\angles{\tgamsig, \alpha \beta \alpha}$, 
are torsion free by Lemma \ref{EasyTor}.

Thus the only remaining subgroups of $\varPi$ to consider are  $\angles{\tgamsig, \alpha}$ and  
$\angles{\tgamsig, \beta \alpha \beta}$,  where
\[
\beta \alpha \beta =\left(\bft_1^{-k_{11}a_1-k_{12}a_2}\bft_2^{-k_{21}a_1-k_{11}a_2}\bft_4^{2b_4+v} ,
\left[ \begin{matrix} 0 & -1 \\ -1 & 0\end{matrix} \right] \right).
\]
By Lemma \ref{torSolof}, $\angles{\tgamsig, \alpha}$ and  $\angles{\tgamsig, \beta \alpha \beta}$ are 
torsion free precisely when their projections to $\Sol$, $\angles{\gamsig, (\bft_1^{a_1} \bft_2^{a_2},A)}$ 
and $\angles{\gamsig,(\bft_1^{-k_{11}a_1 - k_{12}a_2}
\bft_2^{ -k_{21}a_1 - k_{11}a_2} A,B A B)}$, are torsion free.
	
By Proposition \ref{geometricLemma},
we just need to ensure that the appropriate affine maps are fixed point
free on $T^2$, and this occurs precisely when
\begin{align}
\label{C1}
a_1 = \frac 1 2 ,\ \
a_2&\not\equiv \frac{(\sigma_{11}+1)(2 n+1)}{2\sigma_{12}},\\
\label{C2}
\frac{-k_{21}}{2} - k_{11}a_2 &\equiv \frac 12,\\
\label{C3}
\frac{-k_{11}}{2} - k_{12}a_2 &\not\equiv \frac{\sigma_{12}(2 n+1)}{2(\sigma_{11}-1)}.
\end{align}
	
Now we claim that the second part of condition (\ref{C1}) and the condition
(\ref{C3}) are redundant. That is, they follow from (\ref{C2}).

From (\ref{C2}), we have
\begin{align}
\label{C4}
a_2= -\frac {k_{21}+1}{2k_{11}} + \frac p {k_{11}},\ p\in\bbz.
\end{align}

With $a_1 = \frac 12$ and above $a_2$ with $p=0, \ldots, k_{11}-1$,
using that $\det(K)=-1$ and $K^2=\cals$, one can compute that the remaining
criteria are satisfied. In fact, we compute the term in (\ref{C1})
\begin{align*}
\frac{(\sigma_{11}+1)(2 n+1)}{2\sigma_{12}} &= \frac{(k_{11}^2+k_{12}k_{21}+1)(2n+1)}{4k_{11}k_{12}}\\
&=\frac{2k_{12}k_{21}(2n+1)}{4k_{11}k_{12}}=\frac{k_{21}(2n+1)}{2k_{11}}.
\end{align*}

Now, for some $m \in \bbz$, suppose we had
\[
a_2= \frac{(\sigma_{11}+1)(2 n+1)}{2\sigma_{12}}+m,
\]
as opposed to (\ref{C1}). Then we would have
\[
-\frac {k_{21}+1}{2k_{11}} + \frac p {k_{11}} =
\frac{k_{21}(2n+1)}{2k_{11}} + m.
\]
Clearing up, we get
\[
-1+2p=2k_{21}(n+1)+2m k_{11},
\]
a contradiction for any integers $p,n,m$, as they are of different parity. 
Thus, (\ref{C1}) holds.

For (\ref{C3}), using (\ref{C4}), we get
\begin{align*}
\frac{-k_{11}}{2} - k_{12}a_2 
&= \frac{-k_{11}}{2} - k_{12} \left( -\frac {k_{21}+1}{2k_{11}} + \frac p {k_{11}}   \right) \\
&= \frac{ -k_{11}^2+ k_{12}k_{21} + k_{12} - 2k_{12}p}{2k_{11}} \\
&= \frac{1 + k_{12} - 2k_{12}p} {2k_{11}}.
\end{align*}
Now suppose we had
\[
\frac{-k_{11}}{2} - k_{12}a_2 = \frac{\sigma_{12}(2 n+1)}{2(\sigma_{11}-1)}+m
\]
for some $m \in \bbz$.
Then we would have
\[
\frac{1 + k_{12} - 2k_{12} p} {2k_{11}} 
= \frac{\sigma_{12}(2n+1)}{2(\sigma_{11}-1)}+m = \frac{2k_{11}k_{12}(2 n+1)}{2(k_{11}^2 + k_{12}k_{21}-1)}+m.
\]
Clearing up, we get
\[
{1 - 2k_{12}p} = {2 (n k_{12} +m k_{11}) },
\]
a contradiction for any integers $p,n,m$, as they are of different parity.
Thus, (\ref{C3}) holds automatically.

Consequently, with $a_1 = \frac 12$, 
$a_2= -\frac {k_{21}+1}{2k_{11}} + \frac p {k_{11}}$ for 
$p=0, \ldots, k_{11}-1$, 
and $(\beta\alpha)^4 = \bft_4^{\frac{j}{q}}  (j=1,3)$,
$\varPi$ is torsion free.
\bigskip

\noindent
This completes the proof of Theorem \ref{ClassificationSolof-geometry}.
\end{proof}

\section{Examples}

We can embed $\Sol$ and $\Solof$ into $\aff(3)$ and $\aff(4)$, respectively
so that our $\Sol$ and $\Solof$-orbifolds, $Q \bs \Sol$ and $\varPi \bs \Solof$, have complete affinely flat structures. 
Below we use the embedding $\aff(n) = \bbr^n \rx \gl(n,\bbr) \hra \gl(n+1,\bbr)$. See \cite{miln77-1} for the more general question.

One can check the following correspondence is an injective homomorphism
of Lie groups, $\Solof\lra \aff(4)$,
\begin{equation}
\label{affine-rep}
\left[
\begin{array}{ccc}
1 & e^u x & z \\
0 & e^u & y \\
0 & 0 & 1
\end{array}
\right]
\longmapsto
\left[
\begin{array}{ccccc}
1 & -\frac{1}{2} e^{-u} y & \frac{e^u x}{2} & 0 & z-\frac{x y}{2} \\
0 & e^{-u} & 0 & 0 & x \\
0 & 0 & e^u & 0 & y \\
0 & 0 & 0 & 1 & u \\
0 & 0 & 0 & 0 & 1
\end{array}
\right].
\end{equation}
Moreover, the automorphisms
\[
\left[\begin{matrix}
a&0\\0&d
\end{matrix}\right],\quad
\left[\begin{matrix}
0&b\\ c&0
\end{matrix}\right]\in\aut(\Solof)
\]
can also be embedded as
\[
\left[\begin{matrix}
a d &0 &0 &0 &0\\
0 &a &0 &0 &0\\
0 &0 &d &0 &0\\
0 &0 &0 &1 &0\\
0 &0 &0 &0 &1
\end{matrix}\right],\quad
\left[\begin{matrix}
-b c &0 &0 &0 &0\\
0 &0 &b &0 &0\\
0 &c &0 &0 &0\\
0 &0 &0 &-1 &0\\
0 &0 &0 &0 &1
\end{matrix}\right],
\]
respectively, where $a,b,c,d$ are $\pm 1$.
Note that, if we remove the first row and the first column
from $\aff(4)$, we get a representation of $\Sol$ into $\aff(3)$.

If we write the element $(\bfa,A)\in \Solof\rx D_4$
by the product $\bfa\cdot A$, then the group operation of $\Solof\rx D_4$
is compatible with the matrix product in this affine group.
The action of $A$ on $\bfa$ is by conjugation. That is,
\begin{align*}
(\bfa\cdot A)(\bfb\cdot B)&=\bfa A \bfb B\\
&=\bfa (A \bfb A \inv) \cdot AB\\
&=(\bfa,A) \cdot (\bfb,B)
\end{align*}

We have embedded $\isom(\Solof)$ into $\aff(4)$ in such a way that
any lattice acts on $\bbr^4$ properly discontinuously. Therefore
all of our infra-$\Solof$-orbifolds will have an \emph{affine structure}.
Note that not every nilpotent Lie group admits an affine structure \cite[p. 227]{seif-book}.

With $\cals \in \sltz$, $\tr(\cals)>2$, and appropriate $P$ and $\Delta$, so that $P\cals P \inv = \Delta$, we can lift
$\bbz^2\rx_\cals\bbz\subset\bbr^2\rx_\cals\bbr$ to a lattice of $\Solof$ as in the proof
of Theorem \ref{lattice-part-construction}.
The image of our lattice in $\aff(5)$ under the embedding (\ref{affine-rep})
is complicated. When we conjugate it by
\[
P\inv=
\left[
\begin{array}{ccccc}
1 & 0 & 0 & 0 & 0 \\
0 & p_{11} &p_{12} &0 &0\\
0 & p_{21} &p_{22} &0 &0\\
0 & 0 &0 &1 &0\\
0 & 0 &0 &0 &1
\end{array}
\right]\inv,
\]
we get a much better representation of the group as shown below.
Note that $c_3$ will have no effect on the presentation of our lattice.
Since $\det(P)=1$, $[\bft_1,\bft_2]=\bft_4$.

\begin{align*}
\bfe_1&=\left(\left[\begin{matrix} 1\\ 0 \end{matrix} \right],0\right)
\longmapsto
\bft_1=
\left[\begin{matrix}
1 & p_{11} &c_1\\
0 & 1 &p_{21}\\
0 & 0 &1 \end{matrix}
\right]
\longmapsto
\left[
\begin{array}{ccccc}
1 & 0 & \frac{1}{2} & 0 & {c_1}-\frac{\sigma_{21}}{2\sqrt{T^2-4}}\\
0 & 1 & 0 & 0 & 1 \\
0 & 0 & 1 & 0 & 0 \\
0 & 0 & 0 & 1 & 0 \\
0 & 0 & 0 & 0 & 1
\end{array}
\right]
,\\
\bfe_2&=\left(\left[\begin{matrix} 0\\ 1 \end{matrix} \right],0\right)
\longmapsto
\bft_2=
\left[\begin{matrix}
1 & p_{12} &c_2\\
0 & 1 &p_{22}\\
0 & 0 &1 \end{matrix}
\right]
\longmapsto
\left[
\begin{array}{ccccc}
1 & -\frac{1}{2} & 0 & 0 & {c_2}-\frac{\sigma_{12}}{2\sqrt{T^2-4}}\\
0 & 1 & 0 & 0 & 0 \\
0 & 0 & 1 & 0 & 1 \\
0 & 0 & 0 & 1 & 0 \\
0 & 0 & 0 & 0 & 1
\end{array}
\right],\\
\bfe_3&
=\left(\left[\begin{matrix} 0\\ 0 \end{matrix} \right],1\right)
\longmapsto
\bft_3=
\left[\begin{matrix}
1 & 0 &c_3\\
0 & {\lambda} &0\\
0 & 0 &1 \end{matrix}
\right]
\longmapsto
\left[
\begin{array}{ccccc}
1 & 0 & 0 & 0 & {c_3} \\
0 & \sigma_{11} &\sigma_{12} &0 &0\\
0 & \sigma_{21} &\sigma_{22} &0 &0\\
0 & 0 &0 &1 &\ln(\lambda)\\
0 & 0 &0 &0 &1
\end{array}
\right],\\
&\kern6pt\phantom{=\left(\left[\begin{matrix} 0\\
1 \end{matrix} \right],0\right)
\longmapsto}
\bft_4=
\left[\begin{matrix}
1 & 0&1\\
0 & 1 &0\\
0 & 0 &1 \end{matrix}
\right]
\longmapsto
\left[
\begin{array}{ccccc}
1 & 0 & 0 & 0 & 1 \\
0 & 1 & 0 & 0 & 0 \\
0 & 0 & 1 & 0 & 0 \\
0 & 0 & 0 & 1 & 0 \\
0 & 0 & 0 & 0 & 1
\end{array}
\right],
\end{align*}
where $T=\tr(\cals)$.

\begin{example}[\gr{4} Non-standard lattice]
\label{z4-nonzeroc3-example}
This is an example where $c_3$ can be non-zero
(Theorem \ref{ClassificationSolof-geometry}, case \gr{4}). Here
$A=\matNmp$, so that the holonomy $\Phi= \bbz_4$.

Let
$
\cals=
\left[\begin{matrix}
1&2\\ 2&5
\end{matrix}\right].
$
Then $\lambda=3+2 \sqrt{2}$, and with
\[
P=
\left[
\begin{array}{cc}
-\frac{1}{2} \sqrt{2+\sqrt{2}} & \frac{1}{2} \sqrt{2-\sqrt{2}}\\
-\frac{1}{\sqrt{2 \left(2+\sqrt{2}\right)}} & -\frac{1}{2} \sqrt{2+\sqrt{2}}
\end{array}
\right],
\]
our crystallographic group
$\varPi=\angles{\bft_1,\bft_2,\bft_3, \bft_4^{1/q},\alpha}$, where
$\alpha=(\bft_1^{a_1}\bft_2^{a_2}\bft_4^{a_4},A)\in \Solof\rx\aut(\Solof)$,
has a representation into $\aff(4)$:
\begin{align*}
&\bft_1=\left[
\begin{array}{ccccc}
1 & 0 & \frac{1}{2} & 0 & {\msmall {m_1}-\frac{m_2}{2}-\frac{3}{2}} \\
0 & 1 & 0 & 0 & 1 \\
0 & 0 & 1 & 0 & 0 \\
0 & 0 & 0 & 1 & 0 \\
0 & 0 & 0 & 0 & 1
\end{array}
\right],
\bft_2=\left[
\begin{array}{ccccc}
1 & -\frac{1}{2} & 0 & 0 & {\msmall\frac{1}{2} (-{m_1}-1)} \\
0 & 1 & 0 & 0 & 0 \\
0 & 0 & 1 & 0 & 1 \\
0 & 0 & 0 & 1 & 0 \\
0 & 0 & 0 & 0 & 1
\end{array}
\right],\\
&\bft_3=\left[
\begin{array}{ccccc}
1 & 0 & 0 & 0 & {c_3} \\
0 & 1 & 2 & 0 & 0 \\
0 & 2 & 5 & 0 & 0 \\
0 & 0 & 0 & 1 & {\msmall \ln \left(3+2 \sqrt{2}\right)} \\
0 & 0 & 0 & 0 & 1
\end{array}
\right],
\bft_4=\left[
\begin{array}{ccccc}
1 & 0 & 0 & 0 & 1 \\
0 & 1 & 0 & 0 & 0 \\
0 & 0 & 1 & 0 & 0 \\
0 & 0 & 0 & 1 & 0 \\
0 & 0 & 0 & 0 & 1
\end{array}
\right],\\
&(a,A)=\left[
\begin{array}{ccccc}
1 & -\frac{{a_1}}{2} & -\frac{{a_2}}{2} & 0 &{\msmall \frac{1}{2} (2 {a_4}-{a_2} ({m_1}+1)+{a_1} ({a_2}+2 {m_1}-{m_2}-3))} \\
0 & 0 & 1 & 0 & {a_1} \\
0 & -1 & 0 & 0 & {a_2} \\
0 & 0 & 0 & -1 & 0 \\
0 & 0 & 0 & 0 & 1
\end{array}
\right].
\end{align*}

$\varPi$ has  presentation
\begin{align*}
&[\bft_1,\bft_2]=\bft_4,\ \text{ and $\bft_4$ is central, }
\bft_3\bft_1\bft_3\inv=\bft_1^{} \bft_2^{2} \bft_4^{m_1}, \
\bft_3\bft_2\bft_3\inv=\bft_1^{2} \bft_2^{5} \bft_4^{m_2},\\
&\alpha\bft_1\alpha\inv=\bft_2^{-1} \bft_4^{\frac12(-4-2 a_1+m_1-m_2)},\
\alpha\bft_2\alpha\inv=\bft_1^{}\bft_4^{\frac12(2-2 a_2-3 m_1+m_2)},\\
&\alpha\bft_3\alpha\inv=\bft_1^{-4 a_1+2 a_2}
\bft_2^{2 a_1}
\bft_3^{-1}
\bft_4^{5 a_1^2+2 c_3+a_1 (-5+5 m_1-2 m_2)+a_2 (3-a_2-2 m_1+m_2)},\\
&\alpha\bft_4\alpha\inv=\bft_4^{},\
\alpha^4=\bft_4^{-a_1^2+4 a_4-a_2 (2+a_2+2 m_1)+2 a_1 (-3+a_2+2 m_1-m_2)}.
\end{align*}

Since $
(I-\cals)\inv=
\left[\begin{matrix}
1 &-\tfrac12\\ -\tfrac12& 0
\end{matrix}\right],$ $\coker(I-\cals)=\bbz_2\x\bbz_2 \cong 
\left\langle\xy{\tfrac12}{0},\quad \xy{0}{\tfrac12}\right\rangle$.

Therefore, the equation $(I-\cals)\bfa \equiv \bfzero$
has 4 solutions modulo $\bbz^2$;
\[
\xy{a_1}{a_2}
=
\tfrac{1}{2}\xy{0}{0},\
\tfrac{1}{2}\xy{1}{0},\
\tfrac{1}{2}\xy{0}{1},\
\tfrac{1}{2}\xy{1}{1}.
\]
Recall that we had no other conditions on $\bfa$ in Theorem \ref{ClassificationSolof-geometry} case \gr{4}.
The coboundary is
\[
\im(I-\vara)=
\im
\left[\begin{matrix}
1&-1\\ 1&1
\end{matrix}\right]
=
\left\{\xy{0}{0},\xy{\tfrac12}{\tfrac12}\right\}.
\]
Thus, we have only have to consider two cases
\[
\xy{a_1}{a_2}=
\xy{0}{0},\ \
\xy{\tfrac12}{0}.
\]
For simplicity, we shall assume $m_1=m_2=0$.

With $\xy{a_1}{a_2}=\xy{0}{0}$,
$\varPi=\angles{\bft_1,\bft_2,\bft_3, \bft_4^\frac1q,\alpha}$, where
$\alpha=(\bft_4^{a_4},A)$
has presentation
\begin{align*}
&[\bft_1,\bft_2]=\bft_4,\ \text{ and $\bft_4$ is central,}\
\bft_3\bft_1\bft_3\inv=\bft_1^{} \bft_2^{2}, \
\bft_3\bft_2\bft_3\inv=\bft_1^{2} \bft_2^{5},\\
&\alpha\bft_1\alpha\inv=\bft_2^{-1} \bft_4^{-2},\
\alpha\bft_2\alpha\inv=\bft_1^{}\bft_4^{},\
\alpha\bft_3\alpha\inv=\bft_3\inv\bft_4^{2c_3}, \
\alpha\bft_4\alpha\inv=\bft_4^{},\\
&\alpha^4=\bft_4^{4 a_4}.
\end{align*}
The minimum $q$ for $\tgamsig$ is $q=1$. However, to have a torsion free
crystallographic group we must take $q$ to be even, say $q=2$.  Then we have choices
$a_4=0,\frac18,\frac14,\frac38$ and $c_3=0,\frac14$
(any combination of $a_4$ and $c_3$), with the same center.
So, there are 8 distinct groups.
Half of them (with $c_3=0$) have standard lattices, and the rest
(with $c_3=\tfrac14$) have non-standard lattices. When $a_4=\frac{1}{8}$ or $\frac{3}{8}$ (regardless of $c_3$),
 $\varPi$ is torsion free, and $\varPi \bs \Solof$ is an infra-solvmanifold of $\Solof$ 
with $\bbz_4$ holonomy.

With $\xy{a_1}{a_2}=\xy{\frac12}{0}$,
$\varPi=\angles{\bft_1,\bft_2,\bft_3, \bft_4^{\frac1q},\alpha}$, where
$\alpha=(\bft_1^{\frac12}\bft_4^{a_4},A)$
has presentation
\begin{align*}
&[\bft_1,\bft_2]=\bft_4,\ \text{ and $\bft_4$ is central,}\
\bft_3\bft_1\bft_3\inv=\bft_1^{} \bft_2^{2}, \
\bft_3\bft_2\bft_3\inv=\bft_1^{2} \bft_2^{5},\\
&\alpha\bft_1\alpha\inv=\bft_2^{-1} \bft_4^{-\frac52},\
\alpha\bft_2\alpha\inv=\bft_1^{}\bft_4^{},\
\alpha\bft_3\alpha\inv=\bft_1^{-2}\bft_2^{}\bft_3^{-1}
\bft_4^{-\frac54+2c_3},\
\alpha\bft_4\alpha\inv=\bft_4^{},\\
&\alpha^4=\bft_4^{-\frac{13}{4}+4a_4}.
\end{align*}

The minimum $q$ for $\tgamsig$ is $q=2$ (which comes out of
$\alpha\bft_1\alpha\inv=\bft_2^{-1} \bft_4^{-\frac52}$),
and we have choices
$a_4=\frac{1}{16}+\frac{1}{2}\cdot\{0,\frac14,\frac24,\frac34\}=
\frac{1}{16},\frac{3}{16},\frac{5}{16},\frac{7}{16}$ and
$c_3=\frac{1}{8}+\frac{1}{2}\cdot\{0,\frac12\}=\frac{1}{8},\frac{3}{8}$
(any combination of $a_4$ and $c_3$), with the same center.
So, there are 8 distinct groups.

All these groups have non-standard lattices, because no $c_3$ is an
integer multiple of $\frac1q$, $q=2$. When $a_4=\frac{3}{16}$ or $\frac{7}{16}$ (regardless of $c_3$),
$\varPi$ is torsion free, and $\varPi \bs \Solof$ is an infra-solvmanifold of $\Solof$ 
with $\bbz_4$ holonomy.

\end{example}

\begin{example}
[\gr{7\ii} Maximal holonomy]
\label{maxhol}
Even if this has the maximal holonomy group $D_4$, it does not contain
all the possible holonomy actions. For example,
groups of type \gr{6b} or \gr{6b\ii} are not contained in this group.
Let $\Phi=\bbz_4\rx\bbz_2=\angles{A,B}$, where
\begin{align*}
A=\matNpp,\ B=\matPpm, \mbox{ and }
\alpha&=(\bft_1^{a_1}\bft_2^{a_2},A),
\beta=(\bft_3^{\frac12}\bft_4^{b_4},B).
\end{align*}
Our $\cals$ is of the form
$
\cals=n K+I,
$
where
$
K=\left[\begin{matrix}
k_{11} &k_{12}\\ k_{21} & k_{22}
\end{matrix}\right]
$
with $\det(K)=-1$ and  $\tr(K)=n\not=0$.
Now for $\varphi(\alphabar)=\matPpm$, we take $k_{11}=k_{22}$. For example, we need
$
K=
\left[\begin{matrix}
1&2\\ 1&1
\end{matrix}\right],\quad
n=k_{11}+k_{22}=2,\quad
\cals=n K+I=
\left[\begin{matrix}
3&4\\ 2&3
\end{matrix}\right].
$
Then $\lambda=3+2 \sqrt{2}$, and with
$
P=
\left[
\begin{array}{rr}
-\frac{1}{\sqrt[4]{2^3}} &\frac{1}{\sqrt[4]{2}}\\
-\frac{1}{\sqrt[4]{2^3}} &-\frac{1}{\sqrt[4]{2}}
\end{array}
\right],
$
the equations in Lemma \ref{determine-TT} yield
\begin{align*}
c_1=\frac18 (-12+\sqrt{2}+4 m_1-4 m_2),\
c_2=\frac14 (-\sqrt{2}-4 m_1+2 m_2).
\end{align*}
Recall we  can take $c_3=0$ by Theorem \ref{abstract-kernel-theorem}.
Our crystallographic group
$\varPi=\angles{\bft_1,\bft_2,\bft_3, \bft_4^{\frac1q},\alpha, \beta}$
has a representation into $\aff(4)$:
\begin{align*}
&\bft_1=\left[
\begin{array}{ccccc}
1 & 0 & \frac{1}{2} & 0 & \frac{1}{2} ({m_1}-{m_2}-3) \\
0 & 1 & 0 & 0 & 1 \\
0 & 0 & 1 & 0 & 0 \\
0 & 0 & 0 & 1 & 0 \\
0 & 0 & 0 & 0 & 1
\end{array}
\right],\ \
\bft_2=\left[
\begin{array}{ccccc}
1 & -\frac{1}{2} & 0 & 0 & \frac{1}{2} ({m_2}-2 {m_1}) \\
0 & 1 & 0 & 0 & 0 \\
0 & 0 & 1 & 0 & 1 \\
0 & 0 & 0 & 1 & 0 \\
0 & 0 & 0 & 0 & 1
\end{array}
\right],\\
&\bft_3=\left[
\begin{array}{ccccc}
1 & 0 & 0 & 0 & 0 \\
0 & 3 & 4 & 0 & 0 \\
0 & 2 & 3 & 0 & 0 \\
0 & 0 & 0 & 1 & \ln \left(3+2 \sqrt{2}\right)\\
0 & 0 & 0 & 0 & 1
\end{array}
\right],\  \ 
\bft_4=\left[
\begin{array}{ccccc}
1 & 0 & 0 & 0 & 1 \\
0 & 1 & 0 & 0 & 0 \\
0 & 0 & 1 & 0 & 0 \\
0 & 0 & 0 & 1 & 0 \\
0 & 0 & 0 & 0 & 1
\end{array}
\right],\\
&(a,A)=\left[
\begin{array}{ccccc}
-1 & -\frac{a_2}{2} & -\frac{a_1}{2} &
0 & {\msmall \frac{1}{2} (a_1
(a_2+m_1-m_2-3)+a_2
(m_2-2 m_1))} \\
0 & 1 & 0 & 0 & a_1 \\
0 & 0 & -1 & 0 & a_2 \\
0 & 0 & 0 & -1 & 0 \\
0 & 0 & 0 & 0 & 1 \\
\end{array}
\right],\\
&(b,B)=\left[
\begin{array}{ccccc}
-1 & 0&0&0&b_4\\
0 & -1 & -2 & 0 & 0 \\
0 & -1 & -1 & 0 & 0 \\
0 & 0 & 0 & 1 & \frac{1}{2} \ln \left(3+2 \sqrt{2}\right) \\
0 & 0 & 0 & 0 & 1 \\
\end{array}
\right].
\end{align*}

\noindent
We have
\[
\coker(I-\cals)=(\bbz_2)^2
=
\left\{\tfrac12\xy{0}{0},
\tfrac12\xy{0}{1},
\tfrac12\xy{1}{0},
\tfrac12\xy{1}{1}
\right\}.
\]
Now
\[
\varphi(\alphabar)=\matPpm,\quad
\varphi(\betabar)=-K,
\]
yields
\[
I+\varphi(\alphabar)=
\left[\begin{matrix}
2&0\\0&0
\end{matrix}\right],\
I+\varphi(\betabar)=
\left[\begin{matrix}
0&-2\\-1&0
\end{matrix}\right].
\]
Then
$(I+\varphi(\alphabar))\bfa\equiv\bfzero$ yields $2 a_1\equiv 0$, which is not
a new condition.
We therefore have 4 choices for $\bfa$,
\[
\xy{a_1}{a_2}
=\tfrac12\xy{0}{0},\
\tfrac12\xy{0}{1},\
\tfrac12\xy{1}{0},\
\tfrac12\xy{1}{1}.
\]
The coboundary $\im(I-\vara)$ yields the trivial group,
and hence there are 4 distinct choices for $\bfa$.
The group $\varPi$ has a presentation
\begin{align*}
\hskip24pt
&[\bft_1,\bft_2]=\bft_4,\
[\bft_i,\bft_4]=1\ (i=1,2,3),\\
&\bft_3\bft_1\bft_3\inv=\bft_1^{3} \bft_2^{2} \bft_4^{m_1}, \
\bft_3\bft_2\bft_3\inv=\bft_1^{4} \bft_2^{3}\bft_4^{m_2},\\
&\alpha\bft_1\alpha\inv=\bft_1^{}\bft_4^{3-a_2-m_1+m_2},\
\alpha\bft_2\alpha\inv=\bft_2^{-1}\bft_4^{-a_1},\\
&\alpha\bft_3\alpha\inv=\bft_1^{-2 a_1+4 a_2}\bft_2^{2 (a_1-a_2)}\bft_3^{-1}
\bft_4^{3 a_1^2-a_1 (3+6 a_2-3 m_1+2 m_2)+a_2 (6+2 a_2-4 m_1+3 m_2)},\\
&\alpha\bft_4\alpha\inv=\bft_4^{-1},\\
&\beta\bft_1\beta\inv=\bft_1^{-1}\bft_2^{-1}\bft_4^{\frac12(-1-2 m_1+m_2)},\
\beta\bft_2\beta\inv=\bft_1^{-2}\bft_2^{-1}\bft_4^{-4+m_1-m_2},\\
&\beta\bft_3\beta\inv=\bft_3^{},\
\beta\bft_4\beta\inv=\bft_4^{-1},\\
&\alpha^2=\bft_1^{2 a_1}\bft_4^{-a_1 (-3+a_2+m_1-m_2)},\\
&\beta^2=\bft_3^{},\\
&(\alpha \beta)^4=\bft_4^{
-4 b_4+a_1^2+ 4 a_1 a_2 +2 a_2^2 -2 a_1(3-m_1 +m_2)-2 a_2(2m_1-m_2)}.
\end{align*}
Of the four choices for $\bfa$,
only $a_1=\frac12$, $a_2=0$ can yield a torsion free group, and the other three choices always yield a group with torsion:
\begin{alignat*}{2}
\xy{a_1}{a_2}&=\xy{0}{0},\xy{0}{\tfrac12}:\ &&\quad
\alpha^2=\id.\\
\xy{a_1}{a_2}&=\xy{\tfrac12}{\tfrac12}:
&&\quad\Big(\bft_2^{-1}(\alpha\beta)^2\alpha\Big)^2=\id\\
\xy{a_1}{a_2}&=\xy{\frac12}{0}:\  a_2 \equiv 
{\msmall - \frac{k_{21}+1}{2k_{11}} = -1}&&\quad
\end{alignat*}
Let us take $m_1=m_2=0$. When $a_1=\frac 1 2$, $a_2=0$, $q=4$ (minimum), 
$b_4$ takes values $\frac{j}{16}$, $0\leq j\le 3$.
When $b_4=\frac{1}{16}$ or $\frac{3}{16}$, $\varPi$ has torsion.
However, 
when $b_4=0$ or $\frac{2}{16}$, $\varPi$ is torsion free when
$\xy{a_1}{a_2}=\xy{\frac12}{0}$, because all criteria of Theorem \ref{ClassificationSolof-geometry} case \gr{7\ii} are satisfied. In this case, $\varPi \bs \Solof$ is an infra-solvmanifold of $\Solof$ with maximal holonomy $D_4$.
\end{example}

\bibliographystyle{amsplain}

\begin{thebibliography}{20}

\bibitem{4DBieberbachGroup}
H. Brown, R. Bulow, J. Neubuser, H. Wondratschek, and H. Zassenhaus,
{\bf Crystallographic Groups of Four-Dimensional Space}, A.M.S., Wiley Monographs in Crystallography (1978)

\bibitem{Brown}
K. Brown,
{\bf Cohomology of Groups},
GTM 87, Springer-Verlag New York, 1982.

\bibitem{CampTrouy}
 J.T. Campbell and E.C. Trouy,
\emph{When are Two Elements of $\gltz$ Similar?},
{Linear Algebra and its Appl.} {157} (1991), 175-184.


\bibitem{Cobb}
R. Cobb, \emph{Infrasolvmanifolds of Dimension Four},
PhD thesis, The University of Sydney (1999).

\bibitem{Dekimpe}
K. Dekimpe,
{\bf Almost-Bieberbach Groups:  Affine and Polynomial Structures},
Lecture Notes in Mathematics, Springer-Verlag, (1996).

\bibitem{DLR}
K. Dekimpe, K.B. Lee, and F. Raymond,
\emph{Bieberbach theorems for solvable  Lie groups},
Asian J. Math. (2001), no.3, 499--508.

\bibitem{HL}
K. Y. Ha and J. B. Lee, \emph{Crystallographic groups of Sol}, (2013) {Math. Nachr.}, 286: 1614–1667.

\bibitem{hillman1}
J.A. Hillman,
\emph{Geometries and infra-solvmanifolds in dimension 4},
Geom. Dedicata (2007) 129:57--72.

\bibitem{hillman2}
J.A. Hillman,
{\bf Four-Manifolds, Geometries, and Knots},
GT Monographs 5, Geometry and Topology Publications, (2002).

\bibitem{hillman3}
J.A. Hillman,
\emph{$\Sol \x \bbe^1$-Manifolds},
arXiv:1304.2436v2 [math.GT], 18 Apr 2013.

\bibitem{seif-book}
K.B. Lee and F. Raymond,
{\bf Seifert Fiberings}, A.M.S., Mathematical Surveys and Monographs,
vol 166 (2010).


\bibitem{maclane}
S. Mac~Lane, {\bf Homology}, Die {G}rundlehren der {M}ath.
{W}issenschaften, vol. 114, Springer--Verlag Berlin Heidelberg New York,
1975.

\bibitem{miln77-1}
{J. Milnor},
\emph{On fundamental groups of complete affinely flat manifolds},
{Adv. Math.} 25 (1977), {178-187}.

\bibitem{mostow}
G.D. Mostow,
\emph{Self-adjoint groups},
{The Annals of Mathematics, Second Series}, Vol. 62, No. 1 (Jul., 1955)

\bibitem{rag}
M.~S. Raghunathan,
{\it Discrete {S}ubgroups of {L}ie {G}roups},
{Ergebnisse der {M}athematik und ihrer {G}renzgebiete},
vol.~68, Springer--{V}erlag, 1972.

\bibitem{scott}
P. Scott, \emph{The geometries of $3$-manifolds},
Bull. London Math. Soc.
(1983), 15:401--487.

\bibitem{MA}
Wolfram Research,
\emph{Mathematica}, version 9, 2013.


\end{thebibliography}

\end{document}